\documentclass{amsart}
\usepackage{amsfonts}
\usepackage{amsmath}

\newtheorem{theorem}{Theorem}[section]

\newtheorem{lemma}[theorem]{Lemma}

\newtheorem{proposition}[theorem]{Proposition}

\numberwithin{equation}{section}  
 
\begin{document}
\title[Gradient estimate on K\"{a}hler manifolds]{Gradient estimate for
harmonic functions on K\"{a}hler manifolds}
\author{Ovidiu Munteanu}
\email{ovidiu.munteanu@uconn.edu}
\address{Department of Mathematics, University of Connecticut, Storrs, CT
06268, USA}
\author{Lihan Wang}
\email{lihan.wang@uconn.edu}
\address{Department of Mathematics, University of Connecticut, Storrs, CT
06268, USA}
\thanks{
	The first author was partially supported by NSF grant DMS-811845.}
\maketitle

\begin{abstract}
We prove a sharp integral gradient estimate for harmonic functions on
noncompact K\"{a}hler manifolds. As application, we obtain a sharp estimate
for the bottom of spectrum of the $p$-Laplacian and prove a splitting
theorem for manifolds achieving this estimate.
\end{abstract}

\section{Introduction}

This paper studies harmonic functions and spectral information of complete
noncompact manifolds. On a Riemannian manifold $\left( M^{n},g\right) $ the
Laplace operator $\Delta $ acting on functions is essentially self adjoint
and has its $L^{2}$ spectrum contained in $\left[ 0,\infty \right) $.
Properties of harmonic functions are well understood for manifolds with
Ricci curvature bounded from below. If the Ricci curvature is non-negative,
Yau's Liouville theorem \cite{Y} proves that there are no positive harmonic
functions on $M$. Furthermore, there are important works concerning the
space of polynomially growing harmonic functions, for example \cite{CM, L}.

On the other hand, when the Ricci curvature has lower bound $\mathrm{Ric}%
\geq -\left( n-1\right) K$, for some $K>0$, then there may exist positive
harmonic functions. In this case, Yau's gradient estimate asserts that 
\begin{equation}
\left\vert \nabla \ln u\right\vert ^{2}\leq \left( n-1\right) ^{2}K,
\label{Y}
\end{equation}%
for any positive harmonic function $u$ on $M$. This estimate is sharp, as it
can be seen for example on the hyperbolic space $\mathbb{H}^{n}$.

Assume now that $\left( M^{m},g\right) $ is K\"{a}hler, where $m$ is the
complex dimension. On $M$ we consider the Riemannian metric 
\begin{equation*}
ds^{2}:=\mathrm{Re}\left( g_{\alpha \bar{\beta}}dz^{\alpha }d\bar{z}^{\beta
}\right) .
\end{equation*}%
If $\left\{ e_{k}\right\} _{k=1,2m}$ is an orthonormal frame in this metric,
so that $e_{2k}=Je_{2k-1}$, then 
\begin{equation*}
\nu _{\alpha }=\frac{1}{2}\left( e_{2\alpha -1}-Je_{2\alpha }\right)
\end{equation*}%
is a unitary frame, where $\alpha =1,2,..,m$. Assume the Ricci curvature of
this Riemannian metric is bounded below by $\mathrm{Ric}\geq -2\left(
m+1\right) ,$ or equivalently that $R_{\alpha \bar{\beta}}\geq -\left(
m+1\right) \delta _{\alpha \bar{\beta}}$ in the unitary frame $\left\{ \nu
_{\alpha }\right\} _{\alpha =1,m}.$ Then Yau's gradient estimate (\ref{Y})
implies%
\begin{equation}
\left\vert \nabla \ln u\right\vert ^{2}\leq 4m^{2}+2m-2,  \label{YK}
\end{equation}%
for any positive harmonic function $u$ on $M$. For $m\geq 2$ the estimate (%
\ref{YK}) is no longer sharp in the class of complete K\"{a}hler manifolds
with $\mathrm{Ric}\geq -2\left( m+1\right) $. In fact, G. Liu proved \cite%
{Liu} that there exists a constant $\varepsilon \left( m\right) >0$ so that%
\begin{equation*}
\left\vert \nabla \ln u\right\vert ^{2}\leq 4m^{2}+2m-2-\varepsilon \left(
m\right) ,
\end{equation*}%
for any $u>0$ harmonic. In view of known examples, it is an interesting question whether the improved gradient estimate 
\begin{equation}
\left\vert \nabla \ln u\right\vert ^{2}\leq 4m^{2}  \label{Sharp}
\end{equation}%
holds for any positive harmonic function $u$ on a K\"{a}hler manifold with $%
\mathrm{Ric}\geq -2\left( m+1\right)$. There exist positive harmonic functions on  the complex
hyperbolic space $\mathbb{CH}^{m}$ for which equality holds.

In this paper we have established some sharp integral gradient estimates
which present supporting evidence for (\ref{Sharp}).

\begin{theorem}
\label{G}Let $\left( M,g\right) $ be a complete K\"{a}hler manifold of
complex dimension $m,$ with $\mathrm{Ric}\geq -2\left( m+1\right) $. Then
any positive harmonic function $u$ satisfies the integral gradient estimate 
\begin{equation}
\int_{M}u\left\vert \nabla \ln u\right\vert ^{p}\phi ^{2}\leq \left( \left(
2m\right) ^{p}+\epsilon \right) \int_{M}u\phi ^{2}+\frac{c\left( m\right) }{%
\epsilon }\int_{M}u\left\vert \nabla \phi \right\vert ^{2},  \label{p}
\end{equation}%
for any $p\leq 2\left( m+2\right) $, any $\epsilon >0$, and any $\phi \geq 0$
with compact support in $M$.
\end{theorem}

For $p=2$ the estimate was first established by the first author in \cite{M}%
, so our contribution here is to prove it for higher exponents $p\leq 2m+4$.
While in doing this we are inspired by the ideas in \cite{M}, Theorem \ref{G}
will require some delicate new estimates. Let us briefly describe the idea
of proof for (\ref{p}). Recall that Yau's gradient estimate for Riemannian
manifolds uses the maximum principle, the Bochner formula applied to the
function $f=\ln u$, and a clever manipulation of the hessian term $%
\left\vert f_{ij}\right\vert ^{2}$. To get a sharp estimate for K\"{a}hler
manifolds, the hessian term needs to be dealt with differently. To prove (%
\ref{p}) we will use integration by parts and we will estimate the complex
hessian $\left\vert f_{\alpha \bar{\beta}}\right\vert ^{2}$ and the reminder 
$\left\vert f_{\alpha \beta }\right\vert ^{2}$ in different ways. This
strategy seems to break down when $p$\thinspace $>2\left( m+2\right) $,
because some additional terms appear that are difficult to control.

However, we have obtained an integral estimate valid for all exponents $%
p\geq 2$ provided the manifold satisfies an additional assumption. Recall
that the bottom of spectrum of the Laplace operator $\Delta $ is
characterized by%
\begin{equation}
\lambda _{1}\left( M\right) =\inf_{\phi \in C_{0}^{\infty }\left( M\right) }%
\frac{\int_{M}\left\vert \nabla \phi \right\vert ^{2}}{\int_{M}\phi ^{2}}.
\label{V}
\end{equation}%
According to \cite{M}, $\lambda _{1}\left( M\right) \leq m^{2}$ holds on any
complete K\"{a}hler manifold with $\mathrm{Ric}\geq -2\left( m+1\right) $.
This estimate is sharp, being achieved on $\mathbb{CH}^{m}$ and on other
examples \cite{K, LW4}. We have the following result.

\begin{theorem}
\label{GI}Let $\left( M,g\right) $ be a complete K\"{a}hler manifold of
complex dimension $m$, with $\mathrm{Ric}\geq -2\left( m+1\right) $. Assume
in addition that $M$ has maximal bottom of spectrum for the Laplacian, $%
\lambda _{1}\left( M\right) =m^{2}.$ Then any positive harmonic function $u$
satisfies the integral gradient estimate 
\begin{equation*}
\int_{M}u\left\vert \nabla \ln u\right\vert ^{p}\phi ^{2}\leq \left( \left(
2m\right) ^{p}+\epsilon \right) \int_{M}u\phi ^{2}+\frac{c\left( p,m\right) 
}{\epsilon }\int_{M}u\left\vert \nabla \phi \right\vert ^{2},
\end{equation*}%
for any $p\geq 2$, any $\epsilon >0$, and any $\phi \geq 0$ with compact
support in $M$.
\end{theorem}

Theorems \ref{G} and \ref{GI} have applications to spectral estimates. Using
judicious test functions in (\ref{V}) it is possible to obtain upper bound
estimates for the bottom spectrum of the Laplacian. The most natural test
functions $\phi $ in (\ref{V}) are those depending only on distance
function; this eventually needs application of the Laplacian comparison
theorem. Using this approach, Cheng proved the sharp upper bound \cite{C} 
\begin{equation}
\lambda _{1}\left( M\right) \leq \frac{\left( n-1\right) ^{2}}{4}  \label{C}
\end{equation}%
on any Riemannian manifold satisfying $\mathrm{Ric}\geq -\left( n+1\right) $%
. This result was generalized in \cite{Ma} to the bottom spectrum $\lambda
_{1,p}\left( M\right) $ of the $p$-Laplacian%
\begin{equation*}
\Delta _{p}u=\mathrm{div}\left( \left\vert \nabla u\right\vert ^{p-2}\nabla
u\right) \text{,}
\end{equation*}%
which is characterized by 
\begin{equation}
\lambda _{1,p}\left( M\right) =\inf_{\phi \in C_{0}^{\infty }\left( M\right)
}\frac{\int_{M}\left\vert \nabla \phi \right\vert ^{p}}{\int_{M}\left\vert\phi\right\vert ^{p}}.
\label{V1}
\end{equation}

It was proved in \cite{Ma} that 
\begin{equation}
\lambda _{1,p}\left( M\right) \leq \left( \frac{n-1}{p}\right) ^{p}
\label{MP}
\end{equation}%
on any Riemannian manifold with $\mathrm{Ric}\geq -\left( n+1\right) $. This
can be seen as a generalization of Cheng's estimate, because by H\"{o}lder
inequality (\ref{MP}) implies (\ref{C}).

Both (\ref{C}) and (\ref{MP}) are no longer sharp if $\left( M^{m},g\right) $
is K\"{a}hler with Ricci curvature bounded by $\mathrm{Ric}\geq -2\left(
m+1\right) $. Furthermore, the Laplace comparison theorem with comparison
space $\mathbb{CH}^{m}$ fails for only Ricci curvature bounds \cite{Liu}, so
different ideas are now required. In \cite{M} the first author proved the
sharp spectral estimate 
\begin{equation}
\lambda _{1}\left( M\right) \leq m^{2},  \label{K}
\end{equation}%
by using a positive harmonic function (for example, the Green's function) as
a test function in (\ref{V}) and applying the integral gradient estimate in
Theorem \ref{G} for $p=2$.

As application of Theorems \ref{G} and \ref{GI}, we are able to extend (\ref%
{K}) for the $p$ Laplacian.

\begin{theorem}
\label{KE}Let $\left( M,g\right) $ be a K\"{a}hler manifold of complex
dimension $m$, with $\mathrm{Ric}\geq -2\left( m+1\right) $. Then the bottom
spectrum $\lambda _{1,p}\left( M\right) $ of the $p$-Laplacian is bounded by 
\begin{equation}
\lambda _{1,p}\left( M\right) \leq \left( \frac{2m}{p}\right) ^{p},
\label{EK}
\end{equation}%
for any $p\leq 2m+4$. If, moreover, $\left( M,g\right) $ has maximal bottom
of spectrum of the Laplacian, 
\begin{equation*}
\lambda _{1}\left( M\right) =m^{2},
\end{equation*}%
then 
\begin{equation}
\lambda _{1,p}\left( M\right) =\left( \frac{2m}{p}\right) ^{p},  \label{Max}
\end{equation}%
for any $p\geq 2$.
\end{theorem}

Let us note that by H\"{o}lder inequality, the assumption that $\lambda
_{1}\left( M\right) =m^{2}$ implies $\lambda _{1,p}\left( M\right) \geq
\left( \frac{2m}{p}\right) ^{p}$ for any $p\geq 2$. So to prove (\ref{Max})
we showed the converse inequality that $\lambda _{1,p}\left( M\right) \leq
\left( \frac{2m}{p}\right) ^{p}$ for any $p\geq 2$. Hence, (\ref{Max}) can
be rephrased that if $\lambda _{1}\left( M\right) $ is maximal relative to
the Ricci curvature bound, then $\lambda _{1,p}\left( M\right) $ is maximal
as well. The converse of this statement is not known.

Finally, as Theorem \ref{KE} is sharp, we address the equality case. A
remarkable theory developed by P. Li and J. Wang \cite{LW1,LW2,LW3, LW4}
proves rigidity of complete manifolds with more than one end and achieving
maximal bottom of spectrum. As this theory uses harmonic functions
associated to the number of ends of a manifold, it can be applied here to
study rigidity in Theorem \ref{KE}.

\begin{theorem}
\label{Rigid}Let $\left( M,g\right) $ be a K\"{a}hler manifold of complex
dimension $m$, with $\mathrm{Ric}\geq -2\left( m+1\right) $. Assume that $%
p\leq 2m$ and 
\begin{equation*}
\lambda _{1,p}\left( M\right) =\left( \frac{2m}{p}\right) ^{p}.
\end{equation*}%
Then either $M$ has one end or it is diffeomorphic to $\mathbb{R}\times N,$
for a compact $2m-1$ dimensional manifold $N$, and the metric on $M$ is
given by 
\begin{equation*}
ds_{M}^{2}=dt^{2}+e^{-4t}\omega _{2}^{2}+e^{-2t}\left( \omega
_{3}^{2}+...+\omega _{2m}^{2}\right) ,
\end{equation*}%
where $\left\{ \omega _{2},..,\omega _{2m}\right\} $ is an orthonormal
coframe for $N$.
\end{theorem}

The proof of this theorem uses the new estimates obtained in Theorem \ref{G}%
, applied to a harmonic function constructed under the assumption that the
manifold has more than one end. The rigidity is obtained by reading the
equality from the estimates in Theorem \ref{G}. The restriction $p\leq 2m$
is assumed in order to use a result in \cite{SW} that rules out the
existence of two infinite volume ends.

The structure of the paper is as follows. In Section \ref{S1} we prove the
gradient estimate Theorems \ref{G} and \ref{GI}. This is applied in Section %
\ref{S2} to obtain the spectral estimates Theorem \ref{KE}. In Section \ref%
{S3} we study the rigidity result in Theorem \ref{Rigid}.

\section{An integral gradient estimate for harmonic functions\label{S1}}

Let $\left( M,g\right) $ be a K\"{a}hler manifold. On $M$ we consider the
Riemannian metric 
\begin{equation*}
ds^{2}:=\mathrm{Re}\left( g_{\alpha \bar{\beta}}dz^{\alpha }d\bar{z}^{\beta
}\right) .
\end{equation*}%
In this Riemannian metric we have 
\begin{eqnarray*}
\Delta u &=&4u_{\alpha \bar{\alpha}} \\
\left\langle \nabla u,\nabla v\right\rangle &=&2\left( u_{\alpha }v_{\bar{%
\alpha}}+u_{\bar{\alpha}}v_{\alpha }\right) ,
\end{eqnarray*}%
for any two function $u,v$ on $M$. 
Throughout the paper we use Einstein's summation convention. 

 With respect to this metric, Yau's
gradient estimate says that%
\begin{equation*}
\left\vert \nabla \ln u\right\vert ^{2}\leq 4m^{2}+2m-2,
\end{equation*}%
for any positive harmonic function $u$ on $M$. We will prove the following
sharp integral gradient estimate.

\begin{theorem}
\label{Green}Let $\left( M,g\right) $ be a complete K\"{a}hler manifold of
complex dimension $m$, with $\mathrm{Ric}\geq -2\left( m+1\right) $. Then
any positive harmonic function $u$ satisfies the integral gradient estimate 
\begin{equation*}
\int_{M}u\left\vert \nabla \ln u\right\vert ^{p}\phi ^{2}\leq \left( \left(
2m\right) ^{p}+\epsilon \right) \int_{M}u\phi ^{2}+\frac{c\left( m\right) }{%
\epsilon }\int_{M}u\left\vert \nabla \phi \right\vert ^{2},
\end{equation*}%
for any $p\leq 2\left( m+2\right) $, any $\epsilon >0$, and any $\phi \geq 0$
with compact support in $M$. Here $c\left( m\right) >0$ is a constant
depending only on $m$.
\end{theorem}

\begin{proof}
Let $u:M\rightarrow \mathbb{R}$ be a positive harmonic function on a K\"{a}%
hler manifold $\left( M,g\right) $ with $\mathrm{Ric}\geq -2\left(
m+1\right) $. In complex coordinates, this means that 
\begin{equation*}
R_{\alpha \bar{\beta}}\geq -\left( m+1\right) g_{\alpha \bar{\beta}}.
\end{equation*}%
Here and throughout, $\left\{ \nu _{\alpha }\right\} _{\alpha =1,m}$ is a
local unitary frame. Denote with 
\begin{equation}
Q=\left\vert \nabla \ln u\right\vert ^{2}.  \label{Q}
\end{equation}%
Let $\phi $ be a cut-off function with compact support in $M$ and fix any $%
k\geq 0$. To prove this theorem we use a strategy inspired from \cite{M}.
Let us note that the two end example from Theorem \ref{Rigid} admits a
positive harmonic function $w$ so that $\left\vert \nabla w\right\vert =2mw$
and the hessian of $w$ satisfies 
\begin{eqnarray}
w_{\alpha \bar{\beta}} &=&-mg_{\alpha \bar{\beta}}w+w^{-1}w_{\alpha }w_{\bar{%
\beta}}  \label{Model} \\
w_{\alpha \beta } &=&\frac{m+1}{m}w^{-1}w_{\alpha }w_{\beta },  \notag
\end{eqnarray}%
Because for this example $\left\vert w_{\alpha \bar{\beta}}\right\vert
^{2}\neq \left\vert w_{\alpha \beta }\right\vert ^{2}$, we will compute each
expressions separately, using integration by parts. We have%
\begin{eqnarray*}
\int_{M}u^{-1}\left\vert u_{\alpha \bar{\beta}}\right\vert ^{2}Q^{k}\phi
^{2} &=&-\int_{M}u_{\alpha }\left( u^{-1}u_{\bar{\alpha}\beta }Q^{k}\phi
^{2}\right) _{\bar{\beta}} \\
&=&-\int_{M}u^{-1}u_{\bar{\alpha}\beta \bar{\beta}}u_{\alpha }Q^{k}\phi
^{2}+\int_{M}u^{-2}u_{\bar{\alpha}\beta }u_{\alpha }u_{\bar{\beta}}Q^{k}\phi
^{2} \\
&&-k\int_{M}u^{-1}u_{\bar{\alpha}\beta }u_{\alpha }Q_{\bar{\beta}%
}Q^{k-1}\phi ^{2}-\int_{M}u^{-1}u_{\bar{\alpha}\beta }u_{\alpha }\left( \phi
^{2}\right) _{\bar{\beta}}Q^{k}.
\end{eqnarray*}%
Since $u$ is harmonic, we have that $u_{\bar{\alpha}\beta \bar{\beta}}=0$. \
This implies%
\begin{eqnarray}
\int_{M}u^{-1}\left\vert u_{\alpha \bar{\beta}}\right\vert ^{2}Q^{k}\phi
^{2} &=&\int_{M}u^{-2}u_{\bar{\alpha}\beta }u_{\alpha }u_{\bar{\beta}%
}Q^{k}\phi ^{2}  \label{a1} \\
&&-k\int_{M}u^{-1}\mathrm{Re}\left( u_{\bar{\alpha}\beta }u_{\alpha }Q_{\bar{%
\beta}}\right) Q^{k-1}\phi ^{2}  \notag \\
&&-\int_{M}u^{-1}\mathrm{Re}\left( u_{\bar{\alpha}\beta }u_{\alpha }\left(
\phi ^{2}\right) _{\bar{\beta}}\right) Q^{k},  \notag
\end{eqnarray}%
where $\mathrm{Re}\left( z\right) $ denotes the real part of $z$.

Furthermore, since 
\begin{eqnarray*}
Q &=&u^{-2}\left\vert \nabla u\right\vert ^{2} \\
&=&4u^{-2}u_{\gamma }u_{\bar{\gamma}},
\end{eqnarray*}%
it follows that 
\begin{equation*}
Q_{\bar{\beta}}=4u^{-2}u_{\gamma \bar{\beta}}u_{\bar{\gamma}}+4u^{-2}u_{\bar{%
\gamma}\bar{\beta}}u_{\gamma }-2Qu_{\bar{\beta}}u^{-1}.
\end{equation*}

We use this to compute%
\begin{equation}
\mathrm{Re}\left( u_{\bar{\alpha}\beta }u_{\alpha }Q_{\bar{\beta}}\right)
=4u^{-2}\left\vert u_{\bar{\alpha}\beta }u_{\alpha }\right\vert ^{2}+4u^{-2}%
\mathrm{Re}\left( u_{\bar{\alpha}\beta }u_{\alpha }u_{\bar{\gamma}\bar{\beta}%
}u_{\gamma }\right) -2u^{-1}Qu_{\bar{\alpha}\beta }u_{\alpha }u_{\bar{\beta}}
\label{a2}
\end{equation}%
Notice that 
\begin{equation*}
2\mathrm{Re}\left( u_{\bar{\alpha}\beta }u_{\alpha }u_{\bar{\gamma}\bar{\beta%
}}u_{\gamma }\right) =\left\vert u_{\bar{\alpha}\beta }u_{\alpha }+u_{\alpha
\beta }u_{\bar{\alpha}}\right\vert ^{2}-\left\vert u_{\bar{\alpha}\beta
}u_{\alpha }\right\vert ^{2}-\left\vert u_{\alpha \beta }u_{\bar{\alpha}%
}\right\vert ^{2}.
\end{equation*}%
We write 
\begin{eqnarray*}
u_{\bar{\alpha}\beta }u_{\alpha }+u_{\alpha \beta }u_{\bar{\alpha}} &=&\frac{%
1}{4}\left( \left\vert \nabla u\right\vert ^{2}\right) _{\beta } \\
&=&\frac{1}{4}Q_{\beta }u^{2}+\frac{1}{2}Quu_{\beta },
\end{eqnarray*}%
and hence get that 
\begin{eqnarray*}
\left\vert u_{\bar{\alpha}\beta }u_{\alpha }+u_{\alpha \beta }u_{\bar{\alpha}%
}\right\vert ^{2} &=&\frac{1}{16}\left\vert Q_{\beta }u+2Qu_{\beta
}\right\vert ^{2}u^{2} \\
&=&\frac{1}{64}\left\vert \nabla Q\right\vert ^{2}u^{4}+\frac{1}{16}%
Q^{3}u^{4}+\frac{1}{16}Q\left\langle \nabla Q,\nabla u\right\rangle u^{3}.
\end{eqnarray*}%
In conclusion, 
\begin{eqnarray*}
2\mathrm{Re}\left( u_{\bar{\alpha}\beta }u_{\alpha }u_{\bar{\gamma}\bar{\beta%
}}u_{\gamma }\right) &=&\frac{1}{64}\left\vert \nabla Q\right\vert ^{2}u^{4}+%
\frac{1}{16}Q^{3}u^{4}+\frac{1}{16}Q\left\langle \nabla Q,\nabla
u\right\rangle u^{3} \\
&&-\left\vert u_{\bar{\alpha}\beta }u_{\alpha }\right\vert ^{2}-\left\vert
u_{\alpha \beta }u_{\bar{\alpha}}\right\vert ^{2}.
\end{eqnarray*}

Plugging this in (\ref{a2}) it follows that 
\begin{eqnarray}
\mathrm{Re}\left( u_{\bar{\alpha}\beta }u_{\alpha }Q_{\bar{\beta}}\right)
&=&2u^{-2}\left\vert u_{\bar{\alpha}\beta }u_{\alpha }\right\vert
^{2}-2u^{-2}\left\vert u_{\alpha \beta }u_{\bar{\alpha}}\right\vert ^{2}
\label{a2'} \\
&&+\frac{1}{32}\left\vert \nabla Q\right\vert ^{2}u^{2}+\frac{1}{8}%
Q^{3}u^{2}+\frac{1}{8}Q\left\langle \nabla Q,\nabla u\right\rangle u  \notag
\\
&&-2u^{-1}Qu_{\bar{\alpha}\beta }u_{\alpha }u_{\bar{\beta}}.  \notag
\end{eqnarray}%
Similarly, one can also prove the following identity that will be used later 
\begin{eqnarray}
\mathrm{Re}\left( u_{\alpha \beta }u_{\bar{\alpha}}Q_{\bar{\beta}}\right)
&=&2u^{-2}\left\vert u_{\alpha \beta }u_{\bar{\alpha}}\right\vert
^{2}-2u^{-2}\left\vert u_{\bar{\alpha}\beta }u_{\alpha }\right\vert ^{2}
\label{a2''} \\
&&+\frac{1}{32}\left\vert \nabla Q\right\vert ^{2}u^{2}+\frac{1}{8}%
Q^{3}u^{2}+\frac{1}{8}Q\left\langle \nabla Q,\nabla u\right\rangle u  \notag
\\
&&-2u^{-1}Q\mathrm{Re}\left( u_{\alpha \beta }u_{\bar{\alpha}}u_{\bar{\beta}%
}\right) .  \notag
\end{eqnarray}

Plugging (\ref{a2'}) into (\ref{a1}) we obtain 
\begin{eqnarray}
\int_{M}u^{-1}\left\vert u_{\alpha \bar{\beta}}\right\vert ^{2}Q^{k}\phi
^{2} &=&\left( 2k+1\right) \int_{M}u^{-2}u_{\bar{\alpha}\beta }u_{\alpha }u_{%
\bar{\beta}}Q^{k}\phi ^{2}  \label{a3} \\
&&+2k\int_{M}u^{-3}\left\vert u_{\alpha \beta }u_{\bar{\alpha}}\right\vert
^{2}Q^{k-1}\phi ^{2}  \notag \\
&&-2k\int_{M}u^{-3}\left\vert u_{\bar{\alpha}\beta }u_{\alpha }\right\vert
^{2}Q^{k-1}\phi ^{2}  \notag \\
&&-\frac{k}{32}\int_{M}u\left\vert \nabla Q\right\vert ^{2}Q^{k-1}\phi ^{2} 
\notag \\
&&-\frac{k}{8}\int_{M}uQ^{k+2}\phi ^{2}-\frac{k}{8}\int_{M}\left\langle
\nabla Q,\nabla u\right\rangle Q^{k}\phi ^{2}  \notag \\
&&-\int_{M}u^{-1}\mathrm{Re}\left( u_{\bar{\alpha}\beta }u_{\alpha }\left(
\phi ^{2}\right) _{\bar{\beta}}\right) Q^{k}.  \notag
\end{eqnarray}

Integrating by parts, we get that 
\begin{eqnarray*}
-\frac{k}{8}\int_{M}\left\langle \nabla Q,\nabla u\right\rangle Q^{k}\phi
^{2} &=&-\frac{k}{8\left( k+1\right) }\int_{M}\left\langle \nabla
Q^{k+1},\nabla u\right\rangle \phi ^{2} \\
&=&\frac{k}{8\left( k+1\right) }\int_{M}\left\langle \nabla u,\nabla \phi
^{2}\right\rangle Q^{k+1}.
\end{eqnarray*}%
Using this in (\ref{a3}), we conclude that 
\begin{eqnarray}
&&\int_{M}u^{-1}\left\vert u_{\alpha \bar{\beta}}\right\vert ^{2}Q^{k}\phi
^{2}+\frac{k}{32}\int_{M}u\left\vert \nabla Q\right\vert ^{2}Q^{k-1}\phi ^{2}
\label{a4} \\
&=&\left( 2k+1\right) \int_{M}u^{-2}u_{\bar{\alpha}\beta }u_{\alpha }u_{\bar{%
\beta}}Q^{k}\phi ^{2}+2k\int_{M}u^{-3}\left\vert u_{\alpha \beta }u_{\bar{%
\alpha}}\right\vert ^{2}Q^{k-1}\phi ^{2}  \notag \\
&&-2k\int_{M}u^{-3}\left\vert u_{\bar{\alpha}\beta }u_{\alpha }\right\vert
^{2}Q^{k-1}\phi ^{2}-\frac{k}{8}\int_{M}uQ^{k+2}\phi ^{2}+\mathcal{F}_{1}(k), 
\notag
\end{eqnarray}%
where%
\begin{equation*}
\mathcal{F}_{1}(k)=\frac{k}{8\left( k+1\right) }\int_{M}\left\langle \nabla
u,\nabla \phi ^{2}\right\rangle Q^{k+1}-\int_{M}u^{-1}\mathrm{Re}\left( u_{%
\bar{\alpha}\beta }u_{\alpha }\left( \phi ^{2}\right) _{\bar{\beta}}\right)
Q^{k}.
\end{equation*}%
We now proceed similarly and compute 
\begin{eqnarray*}
\int_{M}u^{-1}\left\vert u_{\alpha \beta }\right\vert ^{2}Q^{k}\phi ^{2}
&=&-\int_{M}u_{\bar{\alpha}}\left( u^{-1}u_{\alpha \beta }Q^{k}\phi
^{2}\right) _{\bar{\beta}} \\
&=&-\int_{M}u^{-1}u_{\alpha \beta \bar{\beta}}u_{\bar{\alpha}}Q^{k}\phi
^{2}+\int_{M}u^{-2}u_{\alpha \beta }u_{\bar{\alpha}}u_{\bar{\beta}}Q^{k}\phi
^{2} \\
&&-k\int_{M}u^{-1}u_{\alpha \beta }u_{\bar{\alpha}}Q_{\bar{\beta}%
}Q^{k-1}\phi ^{2}-\int_{M}u^{-1}u_{\alpha \beta }u_{\bar{\alpha}}\left( \phi
^{2}\right) _{\bar{\beta}}Q^{k}.
\end{eqnarray*}%
Note that by Ricci identities,%
\begin{eqnarray*}
-u_{\alpha \beta \bar{\beta}}u_{\bar{\alpha}} &=&-u_{\beta \bar{\beta}\alpha
}u_{\bar{\alpha}}-R_{\alpha \bar{\beta}}u_{\bar{\alpha}}u_{\beta } \\
&\leq &\frac{m+1}{4}\left\vert \nabla u\right\vert ^{2}.
\end{eqnarray*}%
Hence, we get 
\begin{eqnarray*}
\int_{M}u^{-1}\left\vert u_{\alpha \beta }\right\vert ^{2}Q^{k}\phi ^{2}
&\leq &\frac{m+1}{4}\int_{M}uQ^{k+1}\phi ^{2} \\
&&+\int_{M}u^{-2}\mathrm{Re}\left( u_{\alpha \beta }u_{\bar{\alpha}}u_{\bar{%
\beta}}\right) Q^{k}\phi ^{2} \\
&&-k\int_{M}u^{-1}\mathrm{Re}\left( u_{\alpha \beta }u_{\bar{\alpha}}Q_{\bar{%
\beta}}\right) Q^{k-1}\phi ^{2} \\
&&-\int_{M}u^{-1}\mathrm{Re}\left( u_{\alpha \beta }u_{\bar{\alpha}}\left(
\phi ^{2}\right) _{\bar{\beta}}\right) Q^{k}.
\end{eqnarray*}%
By (\ref{a2''}) it follows that 
\begin{eqnarray}
&&\int_{M}u^{-1}\left\vert u_{\alpha \beta }\right\vert ^{2}Q^{k}\phi ^{2}+%
\frac{k}{32}\int_{M}u\left\vert \nabla Q\right\vert ^{2}Q^{k-1}\phi ^{2}
\label{a5} \\
&\leq &\frac{m+1}{4}\int_{M}uQ^{k+1}\phi ^{2}+\left( 2k+1\right)
\int_{M}u^{-2}\mathrm{Re}\left( u_{\alpha \beta }u_{\bar{\alpha}}u_{\bar{%
\beta}}\right) Q^{k}\phi ^{2}  \notag \\
&&+2k\int_{M}u^{-3}\left\vert u_{\bar{\alpha}\beta }u_{\alpha }\right\vert
^{2}Q^{k-1}\phi ^{2}-2k\int_{M}u^{-3}\left\vert u_{\alpha \beta }u_{\bar{%
\alpha}}\right\vert ^{2}Q^{k-1}\phi ^{2}  \notag \\
&&-\frac{k}{8}\int_{M}uQ^{k+2}\phi ^{2}+\frac{k}{8\left( k+1\right) }%
\int_{M}\left\langle \nabla u,\nabla \phi ^{2}\right\rangle Q^{k+1}  \notag
\\
&&-\int_{M}u^{-1}\mathrm{Re}\left( u_{\alpha \beta }u_{\bar{\alpha}}\left(
\phi ^{2}\right) _{\bar{\beta}}\right) Q^{k}.  \notag
\end{eqnarray}%
Finally, from 
\begin{equation*}
u_{\alpha \beta }u_{\bar{\alpha}}+u_{\bar{\alpha}\beta }u_{\alpha }=\frac{1}{%
4}u^{2}Q_{\beta }+\frac{1}{2}uQu_{\beta }
\end{equation*}%
we obtain that 
\begin{eqnarray*}
\mathrm{Re}\left( u_{\alpha \beta }u_{\bar{\alpha}}u_{\bar{\beta}}\right)
&=&-u_{\bar{\alpha}\beta }u_{\alpha }u_{\bar{\beta}} \\
&&+\frac{1}{16}u^{2}\left\langle \nabla Q,\nabla u\right\rangle +\frac{1}{8}%
u^{3}Q^{2}.
\end{eqnarray*}%
Hence, we get 
\begin{eqnarray}
&&\left( 2k+1\right) \int_{M}u^{-2}\mathrm{Re}\left( u_{\alpha \beta }u_{%
\bar{\alpha}}u_{\bar{\beta}}\right) Q^{k}\phi ^{2}  \label{a6'} \\
&=&-\left( 2k+1\right) \int_{M}u^{-2}u_{\bar{\alpha}\beta }u_{\alpha }u_{%
\bar{\beta}}Q^{k}\phi ^{2}  \notag \\
&&-\frac{2k+1}{16\left( k+1\right) }\int_{M}Q^{k+1}\left\langle \nabla
u,\nabla \phi ^{2}\right\rangle  \notag \\
&&+\frac{2k+1}{8}\int_{M}uQ^{k+2}\phi ^{2}.  \notag
\end{eqnarray}%
By (\ref{a5}) we conclude that 
\begin{eqnarray}
&&\int_{M}u^{-1}\left\vert u_{\alpha \beta }\right\vert ^{2}Q^{k}\phi ^{2}+%
\frac{k}{32}\int_{M}u\left\vert \nabla Q\right\vert ^{2}Q^{k-1}\phi ^{2}
\label{a6} \\
&\leq &\frac{m+1}{4}\int_{M}uQ^{k+1}\phi ^{2}-\left( 2k+1\right)
\int_{M}u^{-2}u_{\bar{\alpha}\beta }u_{\alpha }u_{\bar{\beta}}Q^{k}\phi ^{2}
\notag \\
&&+2k\int_{M}u^{-3}\left\vert u_{\bar{\alpha}\beta }u_{\alpha }\right\vert
^{2}Q^{k-1}\phi ^{2}-2k\int_{M}u^{-3}\left\vert u_{\alpha \beta }u_{\bar{%
\alpha}}\right\vert ^{2}Q^{k-1}\phi ^{2}  \notag \\
&&+\frac{k+1}{8}\int_{M}uQ^{k+2}\phi ^{2}+\mathcal{F}_{2}(k),  \notag
\end{eqnarray}

where 
\begin{equation*}
\mathcal{F}_{2}(k)=-\frac{1}{16\left( k+1\right) }\int_{M}\left\langle \nabla
u,\nabla \phi ^{2}\right\rangle Q^{k+1}-\int_{M}u^{-1}\mathrm{Re}\left(
u_{\alpha \beta }u_{\bar{\alpha}}\left( \phi ^{2}\right) _{\bar{\beta}%
}\right) Q^{k}.
\end{equation*}%
Recall that in (\ref{a4}) we proved the following: 
\begin{eqnarray}
&&\int_{M}u^{-1}\left\vert u_{\alpha \bar{\beta}}\right\vert ^{2}Q^{k}\phi
^{2}+\frac{k}{32}\int_{M}u\left\vert \nabla Q\right\vert ^{2}Q^{k-1}\phi ^{2}
\label{a7} \\
&=&\left( 2k+1\right) \int_{M}u^{-2}u_{\bar{\alpha}\beta }u_{\alpha }u_{\bar{%
\beta}}Q^{k}\phi ^{2}+2k\int_{M}u^{-3}\left\vert u_{\alpha \beta }u_{\bar{%
\alpha}}\right\vert ^{2}Q^{k-1}\phi ^{2}  \notag \\
&&-2k\int_{M}u^{-3}\left\vert u_{\bar{\alpha}\beta }u_{\alpha }\right\vert
^{2}Q^{k-1}\phi ^{2}-\frac{1}{8}k\int_{M}uQ^{k+2}\phi ^{2}+\mathcal{F}_{1}(k), 
\notag
\end{eqnarray}%
where%
\begin{equation*}
\mathcal{F}_{1}(k)=\frac{k}{8\left( k+1\right) }\int_{M}\left\langle \nabla
u,\nabla \phi ^{2}\right\rangle Q^{k+1}-\int_{M}u^{-1}\mathrm{Re}\left( u_{%
\bar{\alpha}\beta }u_{\alpha }\left( \phi ^{2}\right) _{\bar{\beta}}\right)
Q^{k}.
\end{equation*}

Adding (\ref{a6}) and (\ref{a7}) implies that 
\begin{eqnarray}
&&\int_{M}u^{-1}\left\vert u_{\alpha \bar{\beta}}\right\vert ^{2}Q^{k}\phi
^{2}+\int_{M}u^{-1}\left\vert u_{\alpha \beta }\right\vert ^{2}Q^{k}\phi ^{2}
\label{a9} \\
&&+\frac{k}{16}\int_{M}u\left\vert \nabla Q\right\vert ^{2}Q^{k-1}\phi ^{2} 
\notag \\
&\leq &\frac{m+1}{4}\int_{M}uQ^{k+1}\phi ^{2}+\frac{1}{8}\int_{M}uQ^{k+2}%
\phi ^{2}  \notag \\
&&+\mathcal{F}_{1}(k)+\mathcal{F}_{2}(k).  \notag
\end{eqnarray}%
Note that\ (\ref{a9}) is exactly the identity that one gets by multiplying
the (Riemannian) Bochner formula 
\begin{equation*}
\frac{1}{2}\Delta \left\vert \nabla u\right\vert ^{2}=\left\vert
u_{ij}\right\vert ^{2}+\mathrm{Ric}\left( \nabla u,\nabla u\right)
\end{equation*}%
by $u^{-1}Q^{k}\phi ^{2}$ and integrating it on $M$. However, inspired by (%
\ref{Model}), we will use different estimates for (\ref{a6}) and (\ref{a7}).

Note that we have the following inequalities%
\begin{equation}
\left\vert u_{\alpha \beta }u_{\bar{\alpha}}\right\vert ^{2}\leq \frac{1}{4}%
\left\vert u_{\alpha \beta }\right\vert ^{2}\left\vert \nabla u\right\vert
^{2}  \label{a12}
\end{equation}%
and 
\begin{eqnarray}
0 &\leq &u^{-1}\left\vert u_{\bar{\alpha}\beta }-u^{-1}u_{\bar{\alpha}%
}u_{\beta }+\frac{1}{4m}Qg_{\bar{\alpha}\beta }u\right\vert ^{2}  \label{a13}
\\
&=&u^{-1}\left\vert u_{\bar{\alpha}\beta }\right\vert ^{2}-2u^{-2}u_{\bar{%
\alpha}\beta }u_{\alpha }u_{\bar{\beta}}+\frac{m-1}{16m}uQ^{2}.  \notag
\end{eqnarray}

Using (\ref{a12}) we have 
\begin{equation*}
\int_{M}u^{-1}\left\vert u_{\alpha \beta }\right\vert ^{2}Q^{k}\phi ^{2}\geq
4\int_{M}u^{-3}\left\vert u_{\alpha \beta }u_{\bar{\alpha}}\right\vert
^{2}Q^{k-1}\phi ^{2}
\end{equation*}%
and from (\ref{a13}) we get 
\begin{equation}
\int_{M}u^{-1}\left\vert u_{\alpha \bar{\beta}}\right\vert ^{2}Q^{k}\phi
^{2}\geq 2\int_{M}u^{-2}u_{\bar{\alpha}\beta }u_{\alpha }u_{\bar{\beta}%
}Q^{k}\phi ^{2}-\frac{m-1}{16m}\int_{M}uQ^{k+2}\phi ^{2}.  \label{a13'}
\end{equation}

Plugging these two inequalities into (\ref{a9}) implies 
\begin{eqnarray}
&&4\int_{M}u^{-3}\left\vert u_{\alpha \beta }u_{\bar{\alpha}}\right\vert
^{2}Q^{k-1}\phi ^{2}+\frac{k}{16}\int_{M}u\left\vert \nabla Q\right\vert
^{2}Q^{k-1}\phi ^{2}  \label{a14} \\
&\leq &-2\int_{M}u^{-2}u_{\bar{\alpha}\beta }u_{\alpha }u_{\bar{\beta}%
}Q^{k}\phi ^{2}+\frac{m+1}{4}\int_{M}uQ^{k+1}\phi ^{2}  \notag \\
&&+\frac{3m-1}{16m}\int_{M}uQ^{k+2}\phi ^{2}+\mathcal{F}_{1}(k)+\mathcal{F}_{2}(k).
\notag
\end{eqnarray}

We have the following inequality 
\begin{eqnarray*}
0 &\leq &\left\vert u_{\bar{\alpha}\beta }u_{\alpha }-\frac{m-1}{4m}%
uQu_{\beta }\right\vert ^{2} \\
&=&\left\vert u_{\bar{\alpha}\beta }u_{\alpha }\right\vert ^{2}-\frac{m-1}{2m%
}uQu_{\bar{\alpha}\beta }u_{\alpha }u_{\bar{\beta}}+\frac{1}{64}\left( \frac{%
m-1}{m}\right) ^{2}Q^{3}u^{4},
\end{eqnarray*}%
from which we deduce that 
\begin{eqnarray}
&&-2\int_{M}u^{-3}\left\vert u_{\bar{\alpha}\beta }u_{\alpha }\right\vert
^{2}Q^{k-1}\phi ^{2}  \label{a15} \\
&\leq &-\frac{m-1}{m}\int_{M}u^{-2}u_{\bar{\alpha}\beta }u_{\alpha }u_{\bar{%
\beta}}Q^{k}\phi ^{2}  \notag \\
&&+\frac{1}{32}\left( \frac{m-1}{m}\right) ^{2}\int_{M}uQ^{k+2}\phi ^{2}. 
\notag
\end{eqnarray}%
By (\ref{a14}) and (\ref{a15}) we get%
\begin{eqnarray*}
&&2k\int_{M}u^{-3}\left\vert u_{\alpha \beta }u_{\bar{\alpha}}\right\vert
^{2}Q^{k-1}\phi ^{2}-2k\int_{M}u^{-3}\left\vert u_{\bar{\alpha}\beta
}u_{\alpha }\right\vert ^{2}Q^{k-1}\phi ^{2} \\
&\leq &-\frac{k^{2}}{32}\int_{M}u\left\vert \nabla Q\right\vert
^{2}Q^{k-1}\phi ^{2}-\frac{\left( 2m-1\right) k}{m}\int_{M}u^{-2}u_{\bar{%
\alpha}\beta }u_{\alpha }u_{\bar{\beta}}Q^{k}\phi ^{2} \\
&&+\frac{\left( 4m^{2}-3m+1\right) k}{32m^{2}}\int_{M}uQ^{k+2}\phi ^{2}+%
\frac{\left( m+1\right) k}{8}\int_{M}uQ^{k+1}\phi ^{2} \\
&&+\frac{k}{2}\left( \mathcal{F}_{1}(k)+\mathcal{F}_{2}(k)\right) .
\end{eqnarray*}

Using this into (\ref{a7}) we obtain 
\begin{eqnarray}
&&\int_{M}u^{-1}\left\vert u_{\alpha \bar{\beta}}\right\vert ^{2}Q^{k}\phi
^{2}+\frac{k\left( k+1\right) }{32}\int_{M}u\left\vert \nabla Q\right\vert
^{2}Q^{k-1}\phi ^{2}  \label{a16} \\
&\leq &\left( 1+\frac{k}{m}\right) \int_{M}u^{-2}u_{\bar{\alpha}\beta
}u_{\alpha }u_{\bar{\beta}}Q^{k}\phi ^{2}  \notag \\
&&-\frac{\left( 3m-1\right) k}{32m^{2}}\int_{M}uQ^{k+2}\phi ^{2}+\frac{%
\left( m+1\right) k}{8}\int_{M}uQ^{k+1}\phi ^{2}  \notag \\
&&+\left( \frac{k}{2}+1\right) \mathcal{F}_{1}(k)+\frac{k}{2}\mathcal{F}_{2}(k). 
\notag
\end{eqnarray}

Plugging (\ref{a13'}) into (\ref{a16}) it follows that%
\begin{eqnarray}
&&\left( 1-\frac{k}{m}\right) \int_{M}u^{-2}u_{\bar{\alpha}\beta }u_{\alpha
}u_{\bar{\beta}}Q^{k}\phi ^{2}+\frac{k\left( k+1\right) }{32}%
\int_{M}u\left\vert \nabla Q\right\vert ^{2}Q^{k-1}\phi ^{2}  \label{a17} \\
&\leq &\left( \frac{m-1}{16m}-\frac{\left( 3m-1\right) k}{32m^{2}}\right)
\int_{M}uQ^{k+2}\phi ^{2}+\frac{\left( m+1\right) k}{8}\int_{M}uQ^{k+1}\phi
^{2}  \notag \\
&&+\left( \frac{k}{2}+1\right) \mathcal{F}_{1}(k)+\frac{k}{2}\mathcal{F}_{2}(k). 
\notag
\end{eqnarray}%
This holds for any $k\geq 0$. Choosing $k=m$ into (\ref{a17}) implies 
\begin{eqnarray}
\int_{M}uQ^{m+2}\phi ^{2} &\leq &4m^{2}\int_{M}uQ^{m+1}\phi ^{2}  \label{a18}
\\
&&+\frac{16m}{m+1}\left( \left( m+2\right) \mathcal{F}_{1}(m)+m\mathcal{F}
_{2}(m)\right) .  \notag
\end{eqnarray}%
By Young's inequality we have 
\begin{equation*}
4m^{2}Q^{m+1}\leq \frac{m+1}{m+2}Q^{m+2}+\frac{1}{m+2}\left( 4m^{2}\right)
^{m+2}.
\end{equation*}%
In conclusion, (\ref{a18}) implies that 
\begin{equation}
\int_{M}uQ^{m+2}\phi ^{2}\leq \left( 4m^{2}\right) ^{m+2}\int_{M}u\phi ^{2}+%
\frac{16m\left( m+2\right) }{m+1}\mathcal{F},  \label{a19}
\end{equation}

where 
\begin{eqnarray*}
\mathcal{F} &\mathcal{=}&\frac{m\left( 2m+3\right) }{16\left( m+1\right) }%
\int_{M}\left\langle \nabla u,\nabla \phi ^{2}\right\rangle Q^{m+1}-\left(
m+2\right) \int_{M}u^{-1}\mathrm{Re}\left( u_{\bar{\alpha}\beta }u_{\alpha
}\left( \phi ^{2}\right) _{\bar{\beta}}\right) Q^{m} \\
&&-m\int_{M}u^{-1}\mathrm{Re}\left( u_{\alpha \beta }u_{\bar{\alpha}}\left(
\phi ^{2}\right) _{\bar{\beta}}\right) Q^{m}.
\end{eqnarray*}%
We now estimate $\mathcal{F}$ as follows. Recall that (\ref{a9}) proved 
\begin{eqnarray}
&&\int_{M}u^{-1}\left\vert u_{\alpha \bar{\beta}}\right\vert ^{2}Q^{m}\phi
^{2}+\int_{M}u^{-1}\left\vert u_{\alpha \beta }\right\vert ^{2}Q^{m}\phi ^{2}
\label{a20} \\
&&+\frac{m}{16}\int_{M}u\left\vert \nabla Q\right\vert ^{2}Q^{m-1}\phi ^{2} 
\notag \\
&\leq &\frac{m+1}{4}\int_{M}uQ^{m+1}\phi ^{2}+\frac{1}{8}\int_{M}uQ^{m+2}%
\phi ^{2}+\frac{2m-1}{16\left( m+1\right) }\int_{M}\left\langle \nabla
u,\nabla \phi ^{2}\right\rangle Q^{m+1}  \notag \\
&&-\int_{M}u^{-1}\mathrm{Re}\left( u_{\bar{\alpha}\beta }u_{\alpha }\left(
\phi ^{2}\right) _{\bar{\beta}}\right) Q^{m}-\int_{M}u^{-1}\mathrm{Re}\left(
u_{\alpha \beta }u_{\bar{\alpha}}\left( \phi ^{2}\right) _{\bar{\beta}%
}\right) Q^{m}.  \notag
\end{eqnarray}%
By the Cauchy-Schwarz inequality we have 
\begin{eqnarray*}
\int_{M}u^{-1}\left\vert u_{\bar{\alpha}\beta }u_{\alpha }\left( \phi
^{2}\right) _{\bar{\beta}}\right\vert Q^{m} &\leq &\frac{1}{2}%
\int_{M}u^{-1}\left\vert u_{\alpha \bar{\beta}}\right\vert ^{2}Q^{m}\phi ^{2}
\\
&&+\frac{1}{2}\int_{M}u\left\vert \nabla \phi \right\vert ^{2}Q^{m+1},
\end{eqnarray*}%
and 
\begin{eqnarray*}
\int_{M}u^{-1}\left\vert u_{\alpha \beta }u_{\bar{\alpha}}\left( \phi
^{2}\right) _{\bar{\beta}}\right\vert Q^{m} &\leq &\frac{1}{2}%
\int_{M}u^{-1}\left\vert u_{\alpha \beta }\right\vert ^{2}Q^{m}\phi ^{2} \\
&&+\frac{1}{2}\int_{M}u\left\vert \nabla \phi \right\vert ^{2}Q^{m+1}.
\end{eqnarray*}%
By (\ref{a20}) this implies that%
\begin{eqnarray}
&&\frac{1}{2}\int_{M}u^{-1}\left\vert u_{\alpha \bar{\beta}}\right\vert
^{2}Q^{m}\phi ^{2}+\frac{1}{2}\int_{M}u^{-1}\left\vert u_{\alpha \beta
}\right\vert ^{2}Q^{m}\phi ^{2}  \label{a20'} \\
&\leq &\frac{m+1}{4}\int_{M}uQ^{m+1}\phi ^{2}+\frac{1}{8}\int_{M}uQ^{m+2}%
\phi ^{2}  \notag \\
&&+\frac{2m-1}{8\left( m+1\right) }\int_{M}\phi \left\vert \nabla
u\right\vert \left\vert \nabla \phi \right\vert Q^{m+1}+\int_{M}u\left\vert
\nabla \phi \right\vert ^{2}Q^{m+1}.  \notag
\end{eqnarray}%
Using Yau's gradient estimate $Q\leq c\left( m\right) =4m^{2}+2m-2$ it
results that 
\begin{eqnarray}
&&\int_{M}u^{-1}\left\vert u_{\alpha \bar{\beta}}\right\vert ^{2}Q^{m}\phi
^{2}+\int_{M}u^{-1}\left\vert u_{\alpha \beta }\right\vert ^{2}Q^{m}\phi ^{2}
\label{a21} \\
&\leq &c\left( m\right) \int_{M}u\left( \phi ^{2}+\left\vert \nabla \phi
\right\vert ^{2}\right) ,  \notag
\end{eqnarray}%
for some constant $c\left( m\right) $ depending only on dimension $m$. \
Hence, for any $\varepsilon >0$ small enough, we get 
\begin{eqnarray}
&&\int_{M}u^{-1}\left\vert u_{\bar{\alpha}\beta }u_{\alpha }\left( \phi
^{2}\right) _{\bar{\beta}}\right\vert Q^{m}+\int_{M}u^{-1}\left\vert
u_{\alpha \beta }u_{\bar{\alpha}}\left( \phi ^{2}\right) _{\bar{\beta}%
}\right\vert Q^{m}  \label{a22} \\
&\leq &\frac{\varepsilon }{c\left( m\right) }\int_{M}u^{-1}\left\vert
u_{\alpha \bar{\beta}}\right\vert ^{2}Q^{m}\phi ^{2}+\frac{\varepsilon }{%
c\left( m\right) }\int_{M}u^{-1}\left\vert u_{\alpha \beta }\right\vert
^{2}Q^{m}\phi ^{2}  \notag \\
&&+\frac{c\left( m\right) }{\varepsilon }\int_{M}uQ^{m+1}\left\vert \nabla
\phi \right\vert ^{2}  \notag \\
&\leq &\varepsilon \int_{M}u\phi ^{2}+\frac{c\left( m\right) }{\varepsilon }%
\int_{M}u\left\vert \nabla \phi \right\vert ^{2},  \notag
\end{eqnarray}%
where in the last line we used (\ref{a21}).

Using (\ref{a22}) we estimate $\mathcal{F}$ from (\ref{a19}) by 
\begin{equation}
\left\vert \mathcal{F}\right\vert \leq c\left( m\right) \varepsilon
\int_{M}u\phi ^{2}+\frac{c\left( m\right) }{\varepsilon }\int_{M}u\left\vert
\nabla \phi \right\vert ^{2}.  \label{a23}
\end{equation}%
Hence, (\ref{a19}) implies%
\begin{equation}
\int_{M}uQ^{m+2}\phi ^{2}\leq \left( \left( 4m^{2}\right) ^{m+2}+\epsilon
\right) \int_{M}u\phi ^{2}+\frac{c\left( m\right) }{\epsilon }%
\int_{M}u\left\vert \nabla \phi \right\vert ^{2},  \label{a24}
\end{equation}%
where $c\left( m\right) $ depends only on dimension. This proves the theorem
for $p=2\left( m+2\right) $. For $p<2\left( m+2\right) $ this follows
immediately from Young's inequality.
\end{proof}

We now prove that Theorem \ref{Green} can be in fact extended to all values
of $p\geq 2$, provided in addition that $\lambda _{1}\left( M\right) $ is
maximal.

\begin{theorem}
\label{Green1}Let $\left( M,g\right) $ be a complete K\"{a}hler manifold of
complex dimension $m$, with $\mathrm{Ric}\geq -2\left( m+1\right) $. Assume
in addition that $M$ has maximal bottom of spectrum for the Laplacian, 
\begin{equation*}
\lambda _{1}\left( M\right) =m^{2}.
\end{equation*}%
Then any positive harmonic function $u$ satisfies the integral gradient
estimate 
\begin{equation*}
\int_{M}u\left\vert \nabla \ln u\right\vert ^{p}\phi ^{2}\leq \left( \left(
2m\right) ^{p}+\epsilon \right) \int_{M}u\phi ^{2}+\frac{c\left( p,m\right) 
}{\epsilon }\int_{M}u\left\vert \nabla \phi \right\vert ^{2},
\end{equation*}%
for any $p\geq 2$, any $\epsilon >0$, and any $\phi \geq 0$ with compact
support in $M$. Here $c\left( p,m\right) $ depends only on $p$ and $m$.
\end{theorem}

\begin{proof}
Start with the inequality 
\begin{eqnarray*}
0 &\leq &u^{-1}\left\vert u^{-1}u_{\alpha \beta }u_{\bar{\beta}}-m\left(
m+1\right) u_{\alpha }\right\vert ^{2} \\
&=&u^{-3}\left\vert u_{\alpha \beta }u_{\bar{\beta}}\right\vert
^{2}-2m\left( m+1\right) u^{-2}\mathrm{Re}\left( u_{\alpha \beta }u_{\bar{%
\alpha}}u_{\bar{\beta}}\right) +\frac{1}{4}m^{2}\left( m+1\right) ^{2}uQ \\
&\leq &\frac{1}{4}u^{-1}\left\vert u_{\alpha \beta }\right\vert
^{2}Q-2m\left( m+1\right) u^{-2}\mathrm{Re}\left( u_{\alpha \beta }u_{\bar{%
\alpha}}u_{\bar{\beta}}\right) +\frac{1}{4}m^{2}\left( m+1\right) ^{2}uQ.
\end{eqnarray*}%
This implies that for any $k\geq 1$, 
\begin{eqnarray*}
\int_{M}u^{-1}\left\vert u_{\alpha \beta }\right\vert ^{2}Q^{k}\phi ^{2}
&\geq &8m\left( m+1\right) \int_{M}u^{-2}\mathrm{Re}\left( u_{\alpha \beta
}u_{\bar{\alpha}}u_{\bar{\beta}}\right) Q^{k-1}\phi ^{2} \\
&&-m^{2}\left( m+1\right) ^{2}\int_{M}uQ^{k}\phi ^{2}.
\end{eqnarray*}%
Combining this with (\ref{a9}), we get 
\begin{eqnarray}
&&\int_{M}u^{-1}\left\vert u_{\alpha \bar{\beta}}\right\vert ^{2}Q^{k}\phi
^{2}+\frac{k}{16}\int_{M}u\left\vert \nabla Q\right\vert ^{2}Q^{k-1}\phi ^{2}
\label{h1} \\
&\leq &-8m\left( m+1\right) \int_{M}u^{-2}\mathrm{Re}\left( u_{\alpha \beta
}u_{\bar{\alpha}}u_{\bar{\beta}}\right) Q^{k-1}\phi ^{2}  \notag \\
&&+\frac{1}{8}\int_{M}uQ^{k+2}\phi ^{2}+\frac{m+1}{4}\int_{M}uQ^{k+1}\phi
^{2}+m^{2}\left( m+1\right) ^{2}\int_{M}uQ^{k}\phi ^{2}  \notag \\
&&+\mathcal{F}_{1}(k)+\mathcal{F}_{2}(k).  \notag
\end{eqnarray}%
Recall that by (\ref{a6'}) we have
\begin{eqnarray}
\int_{M}u^{-2}\mathrm{Re}\left( u_{\alpha \beta }u_{\bar{\alpha}}u_{\bar{%
\beta}}\right) Q^{k-1}\phi ^{2} &=&-\int_{M}u^{-2}u_{\bar{\alpha}\beta
}u_{\alpha }u_{\bar{\beta}}Q^{k-1}\phi ^{2}  \label{h3} \\
&&+\frac{1}{8}\int_{M}uQ^{k+1}\phi ^{2}  \notag \\
&&-\frac{1}{16k}\int_{M}Q^{k}\left\langle \nabla u,\nabla \phi
^{2}\right\rangle .  \notag
\end{eqnarray}%
Hence, by (\ref{h1}) and (\ref{h3}) it follows that 
\begin{eqnarray*}
&&\int_{M}u^{-1}\left\vert u_{\alpha \bar{\beta}}\right\vert ^{2}Q^{k}\phi
^{2}+\frac{k}{16}\int_{M}u\left\vert \nabla Q\right\vert ^{2}Q^{k-1}\phi ^{2}
\\
&\leq &8m\left( m+1\right) \int_{M}u^{-2}u_{\bar{\alpha}\beta }u_{\alpha }u_{%
\bar{\beta}}Q^{k-1}\phi ^{2} \\
&&+\frac{1}{8}\int_{M}uQ^{k+2}\phi ^{2}+\left( \frac{m+1}{4}-m\left(
m+1\right) \right) \int_{M}uQ^{k+1}\phi ^{2} \\
&&+m^{2}\left( m+1\right) ^{2}\int_{M}uQ^{k}\phi ^{2} \\
&&+\mathcal{F}_{1}(k)+\mathcal{F}_{2}(k)+\frac{m\left( m+1\right) }{2k}%
\int_{M}Q^{k}\left\langle \nabla u,\nabla \phi ^{2}\right\rangle .
\end{eqnarray*}%
Finally, combining this with (\ref{a13'}) implies 
\begin{eqnarray}
&&\int_{M}u^{-2}u_{\bar{\alpha}\beta }u_{\alpha }u_{\bar{\beta}}Q^{k}\phi
^{2}  \label{h4} \\
&\leq &4m\left( m+1\right) \int_{M}u^{-2}u_{\bar{\alpha}\beta }u_{\alpha }u_{%
\bar{\beta}}Q^{k-1}\phi ^{2}  \notag \\
&&+\frac{3m-1}{32m}\int_{M}uQ^{k+2}\phi ^{2}-\frac{m+1}{2}\left( m-\frac{1}{4%
}\right) \int_{M}uQ^{k+1}\phi ^{2}  \notag \\
&&+\frac{1}{2}m^{2}\left( m+1\right) ^{2}\int_{M}uQ^{k}\phi ^{2}  \notag \\
&&+\mathcal{F}_{0}\left( k\right) ,  \notag
\end{eqnarray}%
where 
\begin{equation*}
\mathcal{F}_{0}\left( k\right) =\frac{1}{2}\mathcal{F}_{1}(k)+\frac{1}{2}%
\mathcal{F}_{2}(k)+\frac{m\left( m+1\right) }{4k}\int_{M}Q^{k}\left\langle
\nabla u,\nabla \phi ^{2}\right\rangle ,
\end{equation*}%
and $\mathcal{F}_{1}(k),\mathcal{F}_{2}(k)$ are specified in (\ref{a4}) and (\ref%
{a6}), respectively. Denote with 
\begin{eqnarray}
A_{k} &=&\int_{M}u^{-2}u_{\bar{\alpha}\beta }u_{\alpha }u_{\bar{\beta}%
}Q^{k}\phi ^{2}-\frac{3m-1}{32m}\int_{M}uQ^{k+2}\phi ^{2}  \label{h5} \\
&&+\frac{1}{8}m\left( m+1\right) \int_{M}uQ^{k+1}\phi ^{2}.  \notag
\end{eqnarray}%
We now observe that (\ref{h4}) is equivalent to 
\begin{equation}
A_{k}\leq 4m\left( m+1\right) A_{k-1}+\mathcal{F}_{0}\left( k\right) ,
\label{h6}
\end{equation}%
for any $k\geq 1$. We estimate $\mathcal{F}_{0}\left( k\right) $ as in the
proof of Theorem \ref{Green}. Note that 
\begin{eqnarray}
\mathcal{F}_{0}\left( k\right) &=&\frac{2k-1}{32\left( k+1\right) }%
\int_{M}\left\langle \nabla u,\nabla \phi ^{2}\right\rangle Q^{k+1}
\label{h7} \\
&&+\frac{m\left( m+1\right) }{4k}\int_{M}\left\langle \nabla u,\nabla \phi
^{2}\right\rangle Q^{k}  \notag \\
&&-\frac{1}{2}\int_{M}u^{-1}\mathrm{Re}\left( u_{\bar{\alpha}\beta
}u_{\alpha }\left( \phi ^{2}\right) _{\bar{\beta}}\right) Q^{k}  \notag \\
&&-\frac{1}{2}\int_{M}u^{-1}\mathrm{Re}\left( u_{\alpha \beta }u_{\bar{\alpha%
}}\left( \phi ^{2}\right) _{\bar{\beta}}\right) Q^{k}.  \notag
\end{eqnarray}%
As in (\ref{a21}) we get 
\begin{eqnarray*}
&&\int_{M}u^{-1}\left\vert u_{\alpha \bar{\beta}}\right\vert ^{2}Q^{k}\phi
^{2}+\int_{M}u^{-1}\left\vert u_{\alpha \beta }\right\vert ^{2}Q^{k}\phi ^{2}
\\
&\leq &c\left( k,m\right) \int_{M}u\left( \phi ^{2}+\left\vert \nabla \phi
\right\vert ^{2}\right) ,
\end{eqnarray*}%
which yields similarly to (\ref{a23}) that 
\begin{equation*}
\left\vert \mathcal{F}_{0}\left( k\right) \right\vert \leq \varepsilon
\int_{M}u\phi ^{2}+\frac{c\left( k,m\right) }{\varepsilon }%
\int_{M}u\left\vert \nabla \phi \right\vert ^{2}
\end{equation*}%
Using this in (\ref{h6}) implies that 
\begin{equation*}
A_{k}\leq 4m\left( m+1\right) A_{k-1}+\varepsilon \int_{M}u\phi ^{2}+\frac{%
c\left( k,m\right) }{\varepsilon }\int_{M}u\left\vert \nabla \phi
\right\vert ^{2},
\end{equation*}%
for all $k\geq 1$. Iterating from $k=1,2,3...$ we obtain 
\begin{equation}
A_{k}\leq \left( 4m\left( m+1\right) \right) ^{k}A_{0}+c_{1}\left(
k,m\right) \varepsilon \int_{M}u\phi ^{2}+\frac{c_{1}\left( k,m\right) }{%
\varepsilon }\int_{M}u\left\vert \nabla \phi \right\vert ^{2},  \label{h8}
\end{equation}%
where $c_{1}\left( k,m\right) $ depends only on $k$ and dimension $m$. By (%
\ref{h5}) we see that 
\begin{equation*}
A_{0}=\int_{M}u^{-2}u_{\bar{\alpha}\beta }u_{\alpha }u_{\bar{\beta}}\phi
^{2}-\frac{3m-1}{32m}\int_{M}uQ^{2}\phi ^{2}+\frac{1}{8}m\left( m+1\right)
\int_{M}uQ\phi ^{2}.
\end{equation*}%
Using (\ref{a17}) for $k=0$ it follows that 
\begin{eqnarray}
A_{0} &\leq &-\frac{m+1}{32m}\int_{M}uQ^{2}\phi ^{2}+\frac{1}{8}m\left(
m+1\right) \int_{M}uQ\phi ^{2}  \label{h9} \\
&&+\varepsilon \int_{M}u\phi ^{2}+\frac{c\left( m\right) }{\varepsilon }%
\int_{M}u\left\vert \nabla \phi \right\vert ^{2}.  \notag
\end{eqnarray}%
We now use the assumption that $\lambda _{1}\left( M\right) =m^{2}$ and
obtain%
\begin{eqnarray}
m^{2}\int_{M}u\phi ^{2} &\leq &\int_{M}\left\vert \nabla \left( u^{\frac{1}{2%
}}\phi \right) \right\vert ^{2}  \label{h10} \\
&=&\frac{1}{4}\int_{M}u^{-1}\left\vert \nabla u\right\vert ^{2}\phi ^{2}+%
\frac{1}{2}\int_{M}\left\langle \nabla u,\nabla \phi ^{2}\right\rangle
+\int_{M}u\left\vert \nabla \phi \right\vert ^{2}  \notag \\
&=&\frac{1}{4}\int_{M}uQ\phi ^{2}+\int_{M}u\left\vert \nabla \phi
\right\vert ^{2}.  \notag
\end{eqnarray}%
On the other hand, we have 
\begin{equation*}
\int_{M}uQ\phi ^{2}\leq \frac{1}{8m^{2}}\int_{M}uQ^{2}\phi
^{2}+2m^{2}\int_{M}u\phi ^{2},
\end{equation*}%
which combined with (\ref{h10}) implies 
\begin{equation*}
\int_{M}uQ\phi ^{2}\leq \frac{1}{4m^{2}}\int_{M}uQ^{2}\phi
^{2}+4\int_{M}u\left\vert \nabla \phi \right\vert ^{2}.
\end{equation*}%
Using this in (\ref{h9}) yields 
\begin{equation}
A_{0}\leq \varepsilon \int_{M}u\phi ^{2}+\frac{c\left( m\right) }{%
\varepsilon }\int_{M}u\left\vert \nabla \phi \right\vert ^{2},  \label{h11}
\end{equation}%
for some constant $c$ depending only on dimension. By (\ref{h8}) and (\ref%
{h11}) we conclude 
\begin{equation*}
A_{k}\leq c_{2}\left( k,m\right) \varepsilon \int_{M}u\phi ^{2}+\frac{%
c_{2}\left( k,m\right) }{\varepsilon }\int_{M}u\left\vert \nabla \phi
\right\vert ^{2},
\end{equation*}%
where $c_{2}\left( k,m\right) $ depends only on $k$ and dimension $m$.
Hence, by (\ref{h5}) we have proved that 
\begin{eqnarray}
&&\int_{M}u^{-2}u_{\bar{\alpha}\beta }u_{\alpha }u_{\bar{\beta}}Q^{k}\phi
^{2}  \label{h12} \\
&\leq &\frac{3m-1}{32m}\int_{M}uQ^{k+2}\phi ^{2}-\frac{1}{8}m\left(
m+1\right) \int_{M}uQ^{k+1}\phi ^{2}  \notag \\
&&+c_{2}\left( k,m\right) \varepsilon \int_{M}u\phi ^{2}+\frac{c_{2}\left(
k,m\right) }{\varepsilon }\int_{M}u\left\vert \nabla \phi \right\vert ^{2}, 
\notag
\end{eqnarray}

for any $k\geq 0$. \ Recall that by (\ref{a17}) we have 
\begin{eqnarray*}
\left( \frac{\left( 3m-1\right) k}{32m^{2}}-\frac{m-1}{16m}\right)
\int_{M}uQ^{k+2}\phi ^{2} &\leq &\left( \frac{k}{m}-1\right)
\int_{M}u^{-2}u_{\bar{\alpha}\beta }u_{\alpha }u_{\bar{\beta}}Q^{k}\phi ^{2}
\\
&&+\frac{\left( m+1\right) k}{8}\int_{M}uQ^{k+1}\phi ^{2} \\
&&+c_{3}\left( k,m\right) \varepsilon \int_{M}u\phi ^{2}+\frac{c_{3}\left(
k,m\right) }{\varepsilon }\int_{M}u\left\vert \nabla \phi \right\vert ^{2}.
\end{eqnarray*}%
Combining with (\ref{h12}) it follows that for $k\geq m,$%
\begin{eqnarray*}
\int_{M}uQ^{k+2}\phi ^{2} &\leq &4m^{2}\int_{M}uQ^{k+1}\phi ^{2} \\
&&+c\left( k,m\right) \varepsilon \int_{M}u\phi ^{2}+\frac{c\left(
k,m\right) }{\varepsilon }\int_{M}u\left\vert \nabla \phi \right\vert ^{2},
\end{eqnarray*}%
where $c\left( k,m\right) $ depends only on $k$ and dimension $m.$ This
implies the desired result.
\end{proof}

\section{Spectrum of $p$-Laplacian\label{S2}}

As an application of the integral estimate, we prove a sharp upper bound for
the bottom of the spectrum of 
\begin{equation*}
\Delta _{p}u=\mathrm{div}\left( \left\vert \nabla u\right\vert ^{p-2}\nabla
u\right) \text{.}
\end{equation*}%
It is known that this satisfies 
\begin{equation}
\lambda _{1,p}\left( M\right) \leq \frac{\int_{M}\left\vert \nabla \psi
\right\vert ^{p}}{\int_{M}\psi ^{p}},  \label{b0}
\end{equation}%
for any $\psi \geq 0$ with compact support in $M$. Hence, to obtain an upper
bound for $\lambda _{1,p}\left( M\right) $ we will apply (\ref{b0}) to a
carefully chosen test function $\psi $. For this, let us recall some
relation between $\lambda _{1,p}\left( M\right) $ for $p\geq 2$ and $\lambda
_{1}\left( M\right) =\lambda _{1,2}\left( M\right) $. First, observe that
for any $\phi \ge 0 $ with compact support in $M$, 
\begin{eqnarray*}
\lambda _{1}\left( M\right) \int_{M}\phi ^{p} &=&\lambda _{1}\left( M\right)
\int_{M}\left( \phi ^{\frac{p}{2}}\right) ^{2} \\
&\leq &\int_{M}\left\vert \nabla \phi ^{\frac{p}{2}}\right\vert ^{2} \\
&=&\frac{p^{2}}{4}\int_{M}\left\vert \nabla \phi \right\vert ^{2}\phi ^{p-2}
\\
&\leq &\frac{p^{2}}{4}\left( \int_{M}\left\vert \nabla \phi \right\vert
^{p}\right) ^{\frac{2}{p}}\left( \int_{M}\phi ^{p}\right) ^{\frac{p-2}{p}}.
\end{eqnarray*}%
This proves that 
\begin{equation*}
\left( \frac{4}{p^{2}}\lambda _{1}\left( M\right) \right) ^{\frac{p}{2}%
}\int_{M}\phi ^{p}\leq \int_{M}\left\vert \nabla \phi \right\vert ^{p},
\end{equation*}%
for any $\phi \ge 0$ with compact support in $M$. Hence 
\begin{equation}
\lambda _{1,p}\left( M\right) \geq \left( \frac{4}{p^{2}}\lambda _{1}\left(
M\right) \right) ^{\frac{p}{2}}.  \label{b1}
\end{equation}%
According to a result of Sung-Wang, it is possible to obtain a reversed
inequality, but which is not sharp anymore. By (3.8) in \cite{SW} we know
that 
\begin{equation}
\frac{p}{2}\lambda _{1,p}\left( M\right) \int_{M}\phi ^{2}\leq
\int_{M}\left\vert \nabla \ln v\right\vert ^{p-2}\left\vert \nabla \phi
\right\vert ^{2},  \label{b2}
\end{equation}%
for any $\phi $ with compact support, where $v>0$ is a positive
eigenfunction of the $p$ Laplacian, $\mathrm{div}\left( \left\vert \nabla
v\right\vert ^{p-2}\nabla v\right) =-\lambda _{1,p}\left( M\right) v^{p-1}$.
According to Theorem 2.2 in \cite{SW}, on a complete manifold with \textrm{%
Ric}$\geq -2\left( m+1\right) $, we have a gradient estimate 
\begin{equation*}
\left\vert \nabla \ln v\right\vert \leq \sigma ,
\end{equation*}%
where $\sigma $ is the first positive root of the equation 
\begin{equation}
F\left( \sigma \right) :=\left( p-1\right) \sigma ^{p}-\sqrt{2\left(
m+1\right) \left( 2m-1\right) }\sigma ^{p-1}+\lambda _{1,p}\left( M\right)
=0.  \label{b3}
\end{equation}%
It is easy to see that 
\begin{eqnarray*}
F\left( \frac{1}{p}\sqrt{2\left( m+1\right) \left( 2m-1\right) }\right)
&=&\lambda _{1,p}\left( M\right) -\left( \frac{1}{p}\sqrt{2\left( m+1\right)
\left( 2m-1\right) }\right) ^{p} \\
&\leq &0,
\end{eqnarray*}%
where the last line follows applying (\ref{MP}) for the Ricci curvature
bound $\mathrm{Ric}\geq -2\left( m+1\right) $. This implies 
\begin{equation}
\sigma \leq \frac{\sqrt{2\left( m+1\right) \left( 2m-1\right) }}{p}.
\label{b3'}
\end{equation}%
By (\ref{b2}) and (\ref{b3'}) we conclude that 
\begin{equation}
\lambda _{1}\left( M\right) \geq \frac{p}{2}\lambda _{1,p}\left( M\right)
\left( \frac{p}{\sqrt{2\left( m+1\right) \left( 2m-1\right) }}\right) ^{p-2}.
\label{b4}
\end{equation}

From here we infer in particular that $\lambda _{1,p}\left( M\right) >0$
implies $\lambda _{1}\left( M\right) >0$. It is known that a manifold with
positive bottom of spectrum is non-parabolic, so it admits a positive
minimal Green's function $G$ for the Laplacian. The Green's function $%
G\left( x_{0},x\right) $ with a pole at $x_{0}$ is harmonic on $M\backslash
\left\{ x_{0}\right\} $ and will be used as a test function in (\ref{b0}).
We will prove the following result.

\begin{theorem}
\label{E}Let $\left( M,g\right) $ be a K\"{a}hler manifold of complex
dimension $m$, with $\mathrm{Ric}\geq -2\left( m+1\right) $. Then the bottom
spectrum $\lambda _{1,p}\left( M\right) $ of the $p$-Laplacian is bounded by 
\begin{equation*}
\lambda _{1,p}\left( M\right) \leq \left( \frac{2m}{p}\right) ^{p},
\end{equation*}%
for any $2\leq p\leq 2m+4$.
\end{theorem}

\begin{proof}
It suffices to prove the theorem for $p=2m+4$, as the estimate for smaller $p$ follows from H\"{o}lder inequality. Hence, throughout this proof, $p=2m+4$. 

Let us assume by contradiction that $\lambda _{1,p}\left( M\right) >\left( 
\frac{2m}{p}\right) ^{p}$. Then there exists $\varepsilon >0$ so that 
\begin{equation}
\lambda _{1,p}\left( M\right) \geq \left( \frac{2m}{p}\right)
^{p}+2\varepsilon .  \label{b5}
\end{equation}%
Consider $G\left( x_{0},x\right) $ the positive minimal Green's function,
which exists by (\ref{b4}). Define 
\begin{equation*}
\psi \left( x\right) :=G^{\frac{1}{p}}(x_{0},x)\phi \left( x\right) ,
\end{equation*}%
for a cut-off function $\phi \left( x\right) $ with support in $B\left(
x_{0},2R\right) \backslash B\left( x_{0},1\right) $ given by 
\begin{equation*}
\phi \left( x\right) =\left\{ 
\begin{array}{c}
r\left( x\right) -1 \\ 
1 \\ 
R^{-1}\left( 2R-r\left( x\right) \right)%
\end{array}%
\begin{array}{l}
\text{ on }B\left( x_{0},2\right) \backslash B\left( x_{0},1\right) \\ 
\text{on }B\left( x_{0},R\right) \backslash B\left( x_{0},2\right) \\ 
\text{on }B\left( x_{0},2R\right) \backslash B\left( x_{0},R\right)%
\end{array}%
\right.
\end{equation*}

Note that 
\begin{equation*}
\left\vert \nabla \psi \right\vert \leq G^{\frac{1}{p}}\left\vert \nabla
\phi \right\vert +\frac{1}{p}\left\vert \nabla \ln G\right\vert G^{\frac{1}{p%
}}\phi .
\end{equation*}%
Then it follows that 
\begin{eqnarray*}
\int_{M}\left\vert \nabla \psi \right\vert ^{p} &\leq
&p^{-p}\int_{M}G\left\vert \nabla \ln G\right\vert ^{p}\phi ^{p} \\
&&+\sum_{k=0}^{p-1}\left( 
\begin{array}{c}
p \\ 
k%
\end{array}%
\right) \int_{M}\left( \frac{1}{p}\phi \left\vert \nabla G\right\vert G^{%
\frac{1}{p}-1}\right) ^{k}\left( \left\vert \nabla \phi \right\vert G^{\frac{%
1}{p}}\right) ^{p-k} \\
&\leq &p^{-p}\int_{M}G\left\vert \nabla \ln G\right\vert ^{p}\phi
^{p}+c\left( m\right) \sum_{k=0}^{p-1}\int_{M}G\left\vert \nabla \ln
G\right\vert ^{k}\left\vert \nabla \phi \right\vert ^{p-k}.
\end{eqnarray*}%
Since $\left\vert \nabla \phi \right\vert \leq c$ and by Yau's estimate $%
\left\vert \nabla \ln G\right\vert \leq c\left( m\right) $ on the support of 
$\phi $, we get that 
\begin{equation*}
\int_{M}\left\vert \nabla \psi \right\vert ^{p}\leq
p^{-p}\int_{M}G\left\vert \nabla \ln G\right\vert ^{p}\phi ^{p}+c\left(
m\right) \int_{M}G\left\vert \nabla \phi \right\vert .
\end{equation*}%
From the integral estimate from Theorem \ref{Green} we have that 
\begin{equation*}
\int_{M}G\left\vert \nabla \ln G\right\vert ^{p}\phi ^{p}\leq \left( \left(
2m\right) ^{p}+\varepsilon \right) \int_{M}G\phi ^{p}+\frac{c\left( m\right) 
}{\varepsilon }\int_{M}G\left\vert \nabla \phi \right\vert ,
\end{equation*}%
where $c\left( m\right) $ is a constant depending only on $m$. In
conclusion, we obtain that 
\begin{equation}
\int_{M}\left\vert \nabla \psi \right\vert ^{p}\leq \left( \left( \frac{2m}{p%
}\right) ^{p}+\varepsilon \right) \int_{M}G\phi ^{p}+\frac{c\left( m\right) 
}{\varepsilon }\int_{M}G\left\vert \nabla \phi \right\vert .  \label{b8}
\end{equation}

By (\ref{b0}), (\ref{b5}) and (\ref{b8}) we conclude that 
\begin{equation*}
\varepsilon ^{2}\int_{M}G\phi ^{p}\leq c\int_{M}G\left\vert \nabla \phi
\right\vert .
\end{equation*}%
In particular, this proves that there exists a constant $\varepsilon _{0}>0$
so that 
\begin{eqnarray}
\varepsilon _{0}\int_{B\left( x_{0},R\right) \backslash B\left(
x_{0},2\right) }G\left( x_{0},x\right) dx &\leq &\frac{1}{R}\int_{B\left(
x_{0},2R\right) \backslash B\left( x_{0},R\right) }G\left( x_{0},x\right) dx
\label{b9} \\
&&+c\int_{B\left( x_{0},2\right) \backslash B\left( x_{0},1\right) }G\left(
x_{0},x\right) dx,  \notag
\end{eqnarray}%
for any $R>2$. Recall from \cite{MSW} that 
\begin{equation*}
\frac{1}{C}R\le\int_{B\left( x_{0},R\right) }G\left( x_{0},x\right) dx\leq CR,
\end{equation*}%
for some constant $C>0$ dependent only on $m$ and $\lambda _{1}\left( M\right) 
$. Tis contradicts (\ref{b9}). The theorem is proved. 
\end{proof}

We now obtain (\ref{Max}) by applying Theorem \ref{Green1}.

\begin{theorem}
\label{E'}Let $\left( M,g\right) $ be a K\"{a}hler manifold of complex
dimension $m$, with $\mathrm{Ric}\geq -2\left( m+1\right) $. Assume that $%
\left( M,g\right) $ has maximal bottom of spectrum of the Laplacian, 
\begin{equation*}
\lambda _{1}\left( M\right) =m^{2}.
\end{equation*}%
Then 
\begin{equation*}
\lambda _{1,p}\left( M\right) =\left( \frac{2m}{p}\right) ^{p},
\end{equation*}%
for any $p\geq 2$.
\end{theorem}

\begin{proof}
The proof of $\lambda _{1,p}\left( M\right) \leq \left( \frac{2m}{p}\right)
^{p}$ follows as in Theorem \ref{E}. The converse inequality results from (%
\ref{b1}).
\end{proof}

\section{Rigidity for maximal bottom spectrum\label{S3}}

In this section, we follow a theory developed by P. Li and J. Wang \cite%
{LW1, LW2, LW3, LW4} and study the rigidity of manifolds that achieve the
estimate for the bottom \ spectrum of the $p$-Laplacian in Theorem \ref{E},
and have more than one end. The strategy is to use harmonic functions
constructed in \cite{LT} for manifolds with more than one end, whose
behavior depends on whether the end has finite or infinite volume.

Assuming the manifold has at least two ends, we will first prove that one of
these ends must have finite volume. For a harmonic function $u$ associated
to any two ends of the manifold, where one of them has finite volume,
Theorem \ref{Green} proves a gradient estimate that implies Theorem \ref{E}.
When $\lambda _{1,p}\left( M\right) $ is maximal, one can infer from the
proof of Theorem \ref{Green} that all inequalities used there must turn into
equalities. This will imply the splitting of the manifold topologically into
product of the real line with a compact manifold, and will determine the
metric as well. For this approach to work, it is crucial that the boundary
terms expressed in $\mathcal{F}$ in (\ref{a19}) converge to zero for a
carefully chosen cut-off function. This turns out to be the case eventually.
Although each term in $\mathcal{F}$ does not converge to zero on a given
end, it can be computed explicitly and it yields the same absolute constant
but with different signs on the two ends. Hence, after cancellation we are
able to conclude the rigidity question. It should be noted that this
complication arises only in the K\"{a}hler case (cf. \cite{M}).

First, note that if $\lambda _{1,p}\left( M\right) $ is maximal, then $%
\lambda _{1,q}\left( M\right) $ is also maximal, for any $q\geq p$.
Similarly to (\ref{b1}), this follows from H\"{o}lder inequality. Hence,
throughout this section we will assume that $p=2m$ and $\lambda
_{1,2m}\left( M\right) $ is maximal, which is to say $\lambda _{1,2m}\left(
M\right) =1$. The following result therefore implies Theorem \ref{Rigid}.

\begin{theorem}
\label{R}Let $\left( M,g\right) $ be a K\"{a}hler manifold of complex
dimension $m\geq 2$ and with $\mathrm{Ric}\geq -2\left( m+1\right) $.
Suppose $M$ has maximal bottom of spectrum of the $2m$ Laplacian, 
\begin{equation*}
\lambda _{1,2m}\left( M\right) =1.
\end{equation*}%
Then either $M$ has one end or it is diffeomorphic to $\mathbb{R}\times N,$
for a compact $2m-1$ dimensional manifold $N$, and the metric on $M$ is
given by 
\begin{equation*}
ds_{M}^{2}=dt^{2}+e^{-4t}\omega _{2}^{2}+e^{-2t}\left( \omega
_{3}^{2}+...+\omega _{2m}^{2}\right) ,
\end{equation*}%
where $\left\{ \omega _{2},..,\omega _{2m}\right\} $ is an orthonormal
coframe for $N$.
\end{theorem}

The proof will be done in several steps. From now on we will assume that $M$
satisfies the hypothesis of Theorem \ref{R} and that $M$ has at least two
ends. We first record the following result.

\begin{proposition}
\label{Infinite}Let $\left( M,g\right) $ be a K\"{a}hler manifold of complex
dimension $m\geq 2$ and with $\mathrm{Ric}\geq -2\left( m+1\right) $. Assume
that $M$ has maximal bottom of spectrum of the $2m$ Laplacian, 
\begin{equation*}
\lambda _{1,2m}\left( M\right) =1.
\end{equation*}%
Then 
\begin{equation*}
\lambda _{1}\left( M\right) >\frac{m+1}{2}
\end{equation*}%
and $M$ has only one infinite volume end.
\end{proposition}

\begin{proof}
From (\ref{b4}) we have 
\begin{equation*}
\lambda _{1}\left( M\right) \geq m\left( \frac{2m^{2}}{2m^{2}+m-1}\right)
^{m-1}.
\end{equation*}%
It follows through elementary calculations that 
\begin{equation}
\lambda _{1}\left( M\right) >\frac{m+1}{2}  \label{l}
\end{equation}%
for any $m\geq 2$. Indeed, it can be checked that the function 
\begin{equation*}
f\left( m\right) =\left( m-1\right) \ln \left( \frac{2m^{2}}{2m^{2}+m-1}%
\right) -\ln \left( \frac{m+1}{2m}\right)
\end{equation*}%
is decreasing on $\left[ 6,\infty \right) $ and has positive limit at
infinity. This implies (\ref{l}).

By Theorem B in \cite{LW4}, this proves that there exists only one infinite
volume end.
\end{proof}

Proposition \ref{Infinite} implies that there exists an infinite volume end $%
E$ of $M$, and $F=M\backslash E$ is a finite volume end. According to a
result of Li and Tam \cite{LT}, there exists a positive harmonic function $%
u:M\rightarrow \left( 0,\infty \right) $ with the following behavior at
infinity.

On the infinite volume end $E$ the function $u$ is bounded, $\liminf_{E}u=0 $%
, and $u$ has finite Dirichlet energy, $\int_{E}\left\vert \nabla
u\right\vert ^{2}<\infty $. Moreover, it was proved in Lemma 1.1 of \cite%
{LW1} that there exists a constant $C>0$ so that 
\begin{equation}
\int_{E\backslash B\left( x_{0},R\right) }u^{2}\leq Ce^{-2\sqrt{\lambda
_{1}\left( M\right) }R}.  \label{e}
\end{equation}%
On the finite volume end $F$ the function is unbounded, $\limsup_{F}u=\infty 
$. Moreover, by Theorem 1.4 in \cite{LW1} we have 
\begin{equation}
\mathrm{Vol}\left( F\backslash B\left( x_{0},R\right) \right) \leq ce^{-2%
\sqrt{\lambda _{1}\left( M\right) }R}.  \label{v}
\end{equation}%
The next result follows from Theorem \ref{Green} for $p=2m$ by carefully
keeping track of all boundary terms.

\begin{proposition}
\label{B} Let $\left( M,g\right) $ be a K\"{a}hler manifold of complex
dimension $m\geq 2$ and with $\mathrm{Ric}\geq -2\left( m+1\right) $. Assume
that $M$ has maximal bottom of spectrum of the $2m$ Laplacian, $\lambda
_{1,2m}\left( M\right) =1$. Let $u>0$ be the harmonic function defined above
and $\phi $ a non-negative cut-off function satisfying $\phi +\left\vert
\nabla \phi \right\vert \leq c\left( m\right) $. Denoting $Q=\left\vert
\nabla \ln u\right\vert ^{2}$, we have 
\begin{equation}
\left( 4m^{2}\right) ^{m-1}\int_{M}uQ\phi ^{2m}+\mathcal{R}_{2}\leq
\int_{M}uQ^{m}\phi ^{2m}\leq \left( 4m^{2}\right) ^{m-1}\int_{M}uQ\phi ^{2m}+%
\mathcal{R}_{1}  \label{I}
\end{equation}%
and 
\begin{eqnarray}
\int_{M}u^{-2}u_{\bar{\alpha}\beta }u_{\alpha }u_{\bar{\beta}}Q^{m-2}\phi
^{2m} &\leq &\left( 4m^{2}\right) ^{m-2}\frac{m\left( m-1\right) }{4}%
\int_{M}uQ\phi ^{2m}+\mathcal{R}_{3}  \label{II} \\
\int_{M}u^{-2}u_{\bar{\alpha}\beta }u_{\alpha }u_{\bar{\beta}}Q^{m-2}\phi
^{2m} &\geq &\left( 4m^{2}\right) ^{m-2}\frac{m\left( m-1\right) }{4}%
\int_{M}uQ\phi ^{2m}+\mathcal{R}_{4}  \notag
\end{eqnarray}

where%
\begin{eqnarray*}
\mathcal{R}_{i} &=&a_{i}\left( m\right) \int_{M}\left\langle \nabla u,\nabla
\phi ^{2m}\right\rangle Q^{m-1}+b_{i}\left( m\right) \int_{M}u^{-1}\mathrm{Re%
}\left( u_{\bar{\alpha}\beta }u_{\alpha }\left( \phi ^{2m}\right) _{\bar{%
\beta}}\right) Q^{m-2} \\
&&+c_{i}\left( m\right) \int_{M}u^{-1}\mathrm{Re}\left( u_{\alpha \beta }u_{%
\bar{\alpha}}\left( \phi ^{2m}\right) _{\bar{\beta}}\right)
Q^{m-2}+d_{i}\left( m\right) \int_{M}u\left\vert \nabla \phi \right\vert
^{2},
\end{eqnarray*}%
for some constants $a_{i},b_{i},c_{i}\,,d_{i}$ depending only on $m$.
\end{proposition}

\begin{proof}
To prove the first inequality in (\ref{I}) we use that $\lambda
_{1,2m}\left( M\right) =1$, so 
\begin{equation}
\int_{M}\left( u^{\frac{1}{2m}}\phi \right) ^{2m}\leq \int_{M}\left\vert
\nabla \left( u^{\frac{1}{2m}}\phi \right) \right\vert ^{2m}.  \label{m1}
\end{equation}%
We have that 
\begin{eqnarray*}
\left\vert \nabla \left( u^{\frac{1}{2m}}\phi \right) \right\vert ^{2}
&=&\left\vert u^{\frac{1}{2m}}\nabla \phi +\frac{1}{2m}u^{\frac{1}{2m}}\phi
\nabla \ln u\right\vert ^{2} \\
&=&\frac{1}{4m^{2}}\phi ^{2}u^{\frac{1}{m}}Q+u^{\frac{1}{m}}\left\vert
\nabla \phi \right\vert ^{2}+\frac{1}{m}u^{\frac{1}{m}}\phi \left\langle
\nabla \ln u,\nabla \phi \right\rangle .
\end{eqnarray*}%
We write the right hand side of (\ref{m1}) as 
\begin{eqnarray*}
&&\int_{M}\left\vert \nabla \left( u^{\frac{1}{2m}}\phi \right) \right\vert
^{2m} \\
&=&\int_{M}\left( \frac{1}{4m^{2}}\phi ^{2}u^{\frac{1}{m}}Q+u^{\frac{1}{m}%
}\left\vert \nabla \phi \right\vert ^{2}+\frac{1}{m}u^{\frac{1}{m}}\phi
\left\langle \nabla \ln u,\nabla \phi \right\rangle \right) ^{m} \\
&=&\sum_{j=0}^{m}\left( 
\begin{array}{c}
m \\ 
j%
\end{array}%
\right) \int_{M}\left( \frac{1}{4m^{2}}\phi ^{2}u^{\frac{1}{m}}Q\right)
^{m-j}\left( u^{\frac{1}{m}}\left\vert \nabla \phi \right\vert ^{2}+\frac{1}{%
m}u^{\frac{1}{m}}\phi \left\langle \nabla \ln u,\nabla \phi \right\rangle
\right) ^{j} \\
&=&\frac{1}{\left( 4m^{2}\right) ^{m}}\int_{M}uQ^{m}\phi ^{2m} \\
&&+m\int_{M}\left( \frac{1}{4m^{2}}\phi ^{2}u^{\frac{1}{m}}Q\right)
^{m-1}\left( u^{\frac{1}{m}}\left\vert \nabla \phi \right\vert ^{2}+\frac{1}{%
m}u^{\frac{1}{m}}\phi \left\langle \nabla \ln u,\nabla \phi \right\rangle
\right) \\
&&+\sum_{j=2}^{m}\left( 
\begin{array}{c}
m \\ 
j%
\end{array}%
\right) \int_{M}\left( \frac{1}{4m^{2}}\phi ^{2}u^{\frac{1}{m}}Q\right)
^{m-j}\left( u^{\frac{1}{m}}\left\vert \nabla \phi \right\vert ^{2}+\frac{1}{%
m}u^{\frac{1}{m}}\phi \left\langle \nabla \ln u,\nabla \phi \right\rangle
\right) ^{j}.
\end{eqnarray*}%
By Yau's estimate $Q\leq c\left( m\right) $, and by hypothesis $\phi $
satisfies $\phi +\left\vert \nabla \phi \right\vert \leq c\left( m\right) $.
So for $j\geq 2$ we have 
\begin{equation*}
\left( \frac{1}{4m^{2}}\phi ^{2}u^{\frac{1}{m}}Q\right) ^{m-j}\left( u^{%
\frac{1}{m}}\left\vert \nabla \phi \right\vert ^{2}+\frac{1}{m}u^{\frac{1}{m}%
}\phi \left\vert \left\langle \nabla \ln u,\nabla \phi \right\rangle
\right\vert \right) ^{j}\leq c\left( m\right) \left\vert \nabla \phi
\right\vert ^{2}u.
\end{equation*}%
Hence, it follows that 
\begin{eqnarray}
\int_{M}\left\vert \nabla \left( u^{\frac{1}{2m}}\phi \right) \right\vert
^{2m} &\leq &\frac{1}{\left( 4m^{2}\right) ^{m}}\int_{M}uQ^{m}\phi ^{2m}
\label{m2} \\
&&+\frac{1}{\left( 4m^{2}\right) ^{m-1}}\int_{M}\phi ^{2m-1}\left\langle
\nabla u,\nabla \phi \right\rangle Q^{m-1}  \notag \\
&&+c\left( m\right) \int_{M}\left\vert \nabla \phi \right\vert ^{2}u.  \notag
\end{eqnarray}%
Consequently, (\ref{m1}) and (\ref{m2}) imply that 
\begin{eqnarray*}
\int_{M}u\phi ^{2m} &\leq &\frac{1}{\left( 4m^{2}\right) ^{m}}%
\int_{M}uQ^{m}\phi ^{2m} \\
&&+\frac{1}{\left( 4m^{2}\right) ^{m-1}}\int_{M}\phi ^{2m-1}\left\langle
\nabla u,\nabla \phi \right\rangle Q^{m-1} \\
&&+c\left( m\right) \int_{M}\left\vert \nabla \phi \right\vert ^{2}u.
\end{eqnarray*}%
Together with Young's inequality 
\begin{equation*}
\left( 4m^{2}\right) ^{m-1}Q\leq \frac{m-1}{m}\left( 4m^{2}\right) ^{m}+%
\frac{1}{m}Q^{m},
\end{equation*}%
this proves 
\begin{eqnarray}
\int_{M}uQ\phi ^{2m} &\leq &\frac{1}{\left( 4m^{2}\right) ^{m-1}}%
\int_{M}uQ^{m}\phi ^{2m}  \label{m3} \\
&&+\frac{2\left( m-1\right) }{\left( 4m^{2}\right) ^{m-1}}%
\int_{M}\left\langle \nabla u,\nabla \phi ^{2m}\right\rangle Q^{m-1}  \notag
\\
&&+c\left( m\right) \int_{M}\left\vert \nabla \phi \right\vert ^{2}u.  \notag
\end{eqnarray}%
This is the first inequality in (\ref{I}).

We now prove the second inequality in (\ref{I}). Recall that by (\ref{a17})
we get setting $k=m-2$ and replacing $\phi $ by $\phi ^{m},$%
\begin{eqnarray}
\int_{M}u^{-2}u_{\bar{\alpha}\beta }u_{\alpha }u_{\bar{\beta}}Q^{m-2}\phi
^{2m} &\leq &-\frac{m^{2}-5m+2}{64m}\int_{M}uQ^{m}\phi ^{2m}  \label{m4} \\
&&+\frac{m\left( m+1\right) \left( m-2\right) }{16}\int_{M}uQ^{m-1}\phi ^{2m}
\notag \\
&&+\frac{m^{2}}{4}\mathcal{F}_{1}(m-2)+\frac{m\left( m-2\right) }{4}\mathcal{F}%
_{2}(m-2),  \notag
\end{eqnarray}

where 
\begin{equation*}
\mathcal{F}_{1}(m-2)=\frac{m-2}{8\left( m-1\right) }\int_{M}\left\langle \nabla
u,\nabla \phi ^{2m}\right\rangle Q^{m-1}-\int_{M}u^{-1}\mathrm{Re}\left( u_{%
\bar{\alpha}\beta }u_{\alpha }\left( \phi ^{2m}\right) _{\bar{\beta}}\right)
Q^{m-2}
\end{equation*}%
and 
\begin{equation*}
\mathcal{F}_{2}(m-2)=-\frac{1}{16\left( m-1\right) }\int_{M}\left\langle \nabla
u,\nabla \phi ^{2m}\right\rangle Q^{m-1}-\int_{M}u^{-1}\mathrm{Re}\left(
u_{\alpha \beta }u_{\bar{\alpha}}\left( \phi ^{2m}\right) _{\bar{\beta}%
}\right) Q^{m-2}.
\end{equation*}
To simplify notation, we subsequently omit the dependency of $\mathcal{F}_{i}$ on $(m-2)$.
 
We estimate the left side of (\ref{m4}) by using (\ref{a6'}) that 
\begin{eqnarray}
&&\int_{M}u^{-2}\mathrm{Re}\left( u_{\alpha \beta }u_{\bar{\alpha}}u_{\bar{%
\beta}}\right) Q^{m-2}\phi ^{2m}  \label{m5} \\
&=&-\int_{M}u^{-2}u_{\bar{\alpha}\beta }u_{\alpha }u_{\bar{\beta}%
}Q^{m-2}\phi ^{2m}  \notag \\
&&-\frac{1}{16\left( m-1\right) }\int_{M}Q^{m-1}\left\langle \nabla u,\nabla
\phi ^{2m}\right\rangle  \notag \\
&&+\frac{1}{8}\int_{M}uQ^{m}\phi ^{2m}.  \notag
\end{eqnarray}%
Moreover, we have the following estimate 
\begin{eqnarray}
&&\int_{M}u^{-2}\mathrm{Re}\left( u_{\alpha \beta }u_{\bar{\alpha}}u_{\bar{%
\beta}}\right) Q^{m-2}\phi ^{2m}  \label{m6} \\
&\leq &\frac{1}{4}\int_{M}\left\vert u_{\alpha \beta }\right\vert
Q^{m-1}\phi ^{2m}  \notag \\
&\leq &\frac{m}{2\left( m+1\right) }\int_{M}u^{-1}\left\vert u_{\alpha \beta
}\right\vert ^{2}Q^{m-2}\phi ^{2m}+\frac{m+1}{32m}\int_{M}uQ^{m}\phi ^{2m}. 
\notag
\end{eqnarray}

By (\ref{a9}) and (\ref{a13}) we get 
\begin{eqnarray*}
\int_{M}u^{-1}\left\vert u_{\alpha \beta }\right\vert ^{2}Q^{m-2}\phi ^{2m}
&\leq &-2\int_{M}u^{-2}u_{\bar{\alpha}\beta }u_{\alpha }u_{\bar{\beta}%
}Q^{m-2}\phi ^{2m}+\frac{m+1}{4}\int_{M}uQ^{m-1}\phi ^{2m} \\
&&+\frac{3m-1}{16m}\int_{M}uQ^{m}\phi ^{2m}+\mathcal{F}_{1}+\mathcal{F}_{2}.
\end{eqnarray*}

We use this in (\ref{m6}) to conclude 
\begin{eqnarray}
&&\int_{M}u^{-2}\mathrm{Re}\left( u_{\alpha \beta }u_{\bar{\alpha}}u_{\bar{%
\beta}}\right) Q^{m-2}\phi ^{2m}  \label{m7} \\
&\leq &-\frac{m}{m+1}\int_{M}u^{-2}u_{\bar{\alpha}\beta }u_{\alpha }u_{\bar{%
\beta}}Q^{m-2}\phi ^{2m}  \notag \\
&&+\frac{m}{8}\int_{M}uQ^{m-1}\phi ^{2m}+\frac{4m^{2}+m+1}{32m\left(
m+1\right) }\int_{M}uQ^{m}\phi ^{2m}  \notag \\
&&+\frac{m}{2\left( m+1\right) }\left( \mathcal{F}_{1}+\mathcal{F}%
_{2}\right) .  \notag
\end{eqnarray}%
Now (\ref{m5}) and (\ref{m7}) imply that%
\begin{eqnarray}
&&\int_{M}u^{-2}u_{\bar{\alpha}\beta }u_{\alpha }u_{\bar{\beta}}Q^{m-2}\phi
^{2m}  \label{m8} \\
&\geq &\frac{3m-1}{32m}\int_{M}uQ^{m}\phi ^{2m}-\frac{m\left( m+1\right) }{8}%
\int_{M}uQ^{m-1}\phi ^{2m}  \notag \\
&&-\frac{m+1}{16\left( m-1\right) }\int_{M}Q^{m-1}\left\langle \nabla
u,\nabla \phi ^{2m}\right\rangle  \notag \\
&&-\frac{m}{2}\left( \mathcal{F}_{1}+\mathcal{F}_{2}\right) .  \notag
\end{eqnarray}%
Combining this with (\ref{m4}) yields 
\begin{eqnarray}
\int_{M}uQ^{m}\phi ^{2m} &\leq &4m^{2}\int_{M}uQ^{m-1}\phi ^{2m}  \label{m9}
\\
&&+\frac{4}{m-1}\int_{M}\left\langle \nabla u,\nabla \phi ^{2m}\right\rangle
Q^{m-1}  \notag \\
&&+16\frac{m\left( m+2\right) }{m+1}\mathcal{F}_{1}+16\frac{m^{2}}{m+1}%
\mathcal{F}_{2}.  \notag
\end{eqnarray}%
We use Young's inequality 
\begin{equation}
4m^{2}Q^{m-1}\leq \frac{m-2}{m-1}Q^{m}+\frac{1}{m-1}\left( 4m^{2}\right)
^{m-1}Q  \label{m10}
\end{equation}%
to conclude from (\ref{m9}) that 
\begin{eqnarray}
\int_{M}uQ^{m}\phi ^{2m} &\leq &\left( 4m^{2}\right) ^{m-1}\int_{M}uQ\phi
^{2m}+4\int_{M}\left\langle \nabla u,\nabla \phi ^{2m}\right\rangle Q^{m-1}
\label{m11} \\
&&+16\frac{m\left( m-1\right) \left( m+2\right) }{m+1}\mathcal{F}_{1}+16%
\frac{m^{2}\left( m-1\right) }{m+1}\mathcal{F}_{2}.  \notag
\end{eqnarray}

This completes the proof of (\ref{I}).

To prove the first inequality in (\ref{II}), we first use (\ref{m10}) into (%
\ref{m4}) to get 
\begin{eqnarray*}
&&\int_{M}u^{-2}u_{\bar{\alpha}\beta }u_{\alpha }u_{\bar{\beta}}Q^{m-2}\phi
^{2m} \\
&\leq &\frac{3m^{2}-7m+6}{64m\left( m-1\right) }\int_{M}uQ^{m}\phi ^{2m}+%
\frac{m\left( m+1\right) \left( m-2\right) }{16\left( m-1\right) }\left(
4m^{2}\right) ^{m-2}\int_{M}uQ\phi ^{2m} \\
&&+\frac{m^{2}}{4}\mathcal{F}_{1}+\frac{m\left( m-2\right) }{4}\mathcal{F}%
_{2}.
\end{eqnarray*}

We plug (\ref{m11}) into this inequality and obtain 
\begin{eqnarray*}
\int_{M}u^{-2}u_{\bar{\alpha}\beta }u_{\alpha }u_{\bar{\beta}}Q^{m-2}\phi
^{2m} &\leq &\left( 4m^{2}\right) ^{m-2}\frac{m\left( m-1\right) }{4}%
\int_{M}uQ\phi ^{2m} \\
&&+\mathcal{R}_{3},
\end{eqnarray*}%
where 
\begin{eqnarray*}
\mathcal{R}_{3} &=&a_{3}\left( m\right) \int_{M}\left\langle \nabla u,\nabla
\phi ^{2m}\right\rangle Q^{m-1}+b_{3}\left( m\right) \int_{M}u^{-1}\mathrm{Re%
}\left( u_{\bar{\alpha}\beta }u_{\alpha }\left( \phi ^{2m}\right) _{\bar{%
\beta}}\right) Q^{m-2} \\
&&+c_{3}\left( m\right) \int_{M}u^{-1}\mathrm{Re}\left( u_{\alpha \beta }u_{%
\bar{\alpha}}\left( \phi ^{2m}\right) _{\bar{\beta}}\right)
Q^{m-2}+d_{3}\left( m\right) \int_{M}\left\vert \nabla \phi \right\vert
^{2}u.
\end{eqnarray*}%
This proves the first inequality in (\ref{II}).

To prove the second inequality in (\ref{II}), we plug (\ref{m10}) into (\ref%
{m8}) and get 
\begin{eqnarray*}
\int_{M}u^{-2}u_{\bar{\alpha}\beta }u_{\alpha }u_{\bar{\beta}}Q^{m-2}\phi
^{2m} &\geq &\frac{2m^{2}-3m+3}{32m\left( m-1\right) }\int_{M}uQ^{m}\phi
^{2m} \\
&&-\frac{m\left( m+1\right) }{8\left( m-1\right) }\left( 4m^{2}\right)
^{m-2}\int_{M}uQ\phi ^{2m} \\
&&-\frac{m+1}{16\left( m-1\right) }\int_{M}Q^{m-1}\left\langle \nabla
u,\nabla \phi ^{2m}\right\rangle \\
&&-\frac{m}{2}\left( \mathcal{F}_{1}+\mathcal{F}_{2}\right) .
\end{eqnarray*}%
Using (\ref{m3}) into this inequality it follows that 
\begin{equation*}
\int_{M}u^{-2}u_{\bar{\alpha}\beta }u_{\alpha }u_{\bar{\beta}}Q^{m-2}\phi
^{2m}\geq \left( 4m^{2}\right) ^{m-2}\frac{m\left( m-1\right) }{4}%
\int_{M}uQ\phi ^{2m}+\mathcal{R}_{4},
\end{equation*}%
for 
\begin{eqnarray*}
\mathcal{R}_{4} &=&a_{4}\left( m\right) \int_{M}\left\langle \nabla u,\nabla
\phi ^{2m}\right\rangle Q^{m-1}+b_{4}\left( m\right) \int_{M}u^{-1}\mathrm{Re%
}\left( u_{\bar{\alpha}\beta }u_{\alpha }\left( \phi ^{2m}\right) _{\bar{%
\beta}}\right) Q^{m-2} \\
&&+c_{4}\left( m\right) \int_{M}u^{-1}\mathrm{Re}\left( u_{\alpha \beta }u_{%
\bar{\alpha}}\left( \phi ^{2m}\right) _{\bar{\beta}}\right)
Q^{m-2}+d_{4}\left( m\right) \int_{M}\left\vert \nabla \phi \right\vert
^{2}u.
\end{eqnarray*}%
This proves the result.
\end{proof}

From the inequality (\ref{I}) in Proposition \ref{B} we see that 
\begin{equation}
\mathcal{R}_{2}\leq \int_{M}uQ^{m}\phi ^{2m}-\left( 4m^{2}\right)
^{m-1}\int_{M}uQ\phi ^{2m}\leq \mathcal{R}_{1},  \label{g1}
\end{equation}%
where for $i\in \left\{ 1,2\right\} $%
\begin{eqnarray*}
\mathcal{R}_{i} &=&a_{i}\left( m\right) \int_{M}\left\langle \nabla u,\nabla
\phi ^{2m}\right\rangle Q^{m-1}+b_{i}\left( m\right) \int_{M}u^{-1}\mathrm{Re%
}\left( u_{\bar{\alpha}\beta }u_{\alpha }\left( \phi ^{2m}\right) _{\bar{%
\beta}}\right) Q^{m-2} \\
&&+c_{i}\left( m\right) \int_{M}u^{-1}\mathrm{Re}\left( u_{\alpha \beta }u_{%
\bar{\alpha}}\left( \phi ^{2m}\right) _{\bar{\beta}}\right)
Q^{m-2}+d_{i}\left( m\right) \int_{M}u\left\vert \nabla \phi \right\vert
^{2},
\end{eqnarray*}%
for some constants $a_{i},b_{i},c_{i}\,,d_{i}$ depending only on $m$. From (%
\ref{g1}) we conclude that there exist constants $\alpha _{i}\left( m\right)
>0$ so that 
\begin{eqnarray}
&&\left\vert \int_{M}uQ^{m}\phi ^{2m}-\left( 4m^{2}\right)
^{m-1}\int_{M}uQ\phi ^{2m}\right\vert  \label{y3} \\
&\leq &\alpha _{1}\left( m\right) \left\vert \int_{M}\left\langle \nabla
u,\nabla \phi ^{2m}\right\rangle Q^{m-1}\right\vert  \notag \\
&&+\alpha _{2}\left( m\right) \left\vert \int_{M}u^{-1}\mathrm{Re}\left( u_{%
\bar{\alpha}\beta }u_{\alpha }\left( \phi ^{2m}\right) _{\bar{\beta}}\right)
Q^{m-2}\right\vert  \notag \\
&&+\alpha _{3}\left( m\right) \left\vert \int_{M}u^{-1}\mathrm{Re}\left(
u_{\alpha \beta }u_{\bar{\alpha}}\left( \phi ^{2m}\right) _{\bar{\beta}%
}\right) Q^{m-2}\right\vert  \notag \\
&&+\alpha _{4}\left( m\right) \int_{M}u\left\vert \nabla \phi \right\vert
^{2}.  \notag
\end{eqnarray}

Our goal is to prove that the right side of (\ref{y3}) converges to zero. As
in \cite{LW3}, we will use $u$ to construct a cut-off function $\phi $ on $M$%
. Denote the level and sublevel sets of $u$ with 
\begin{eqnarray*}
l\left( t\right) &=&\left\{ x\in M:u\left( x\right) =t\right\} \\
L\left( a,b\right) &=&\left\{ x\in M:a<u\left( x\right) <b\right\} .
\end{eqnarray*}%
Note that $l\left( t\right) $ and $L\left( a,b\right) $ may not be compact
for all $t$ and all $a<b$.

For arbitrary numbers $0<\delta \varepsilon <\varepsilon <T<\beta T<\infty $
and $R>1$ we consider the cut-off function 
\begin{equation}
\phi =\left( \chi \psi \right) ^{2}  \label{y4}
\end{equation}%
where $\chi $ is given by 
\begin{equation}
\chi =\left\{ 
\begin{array}{c}
\left( -\ln \delta \right) ^{-1}\left( \ln u-\ln \delta \varepsilon \right) 
\text{ } \\ 
1 \\ 
\left( \ln \beta \right) ^{-1}\left( \ln \left( \beta T\right) -\ln u\right)
\\ 
0%
\end{array}%
\right. 
\begin{array}{l}
\text{on }L\left( \delta \varepsilon ,\varepsilon \right) \\ 
\text{on \ }L\left( \varepsilon ,T\right) \\ 
\text{on \ }L\left( T,\beta T\right) \\ 
\text{otherwise}%
\end{array}
\label{hi}
\end{equation}%
and $\psi \left( r\right) $ is a function with support in $B\left(
x_{0},2R\right) $ so that $\psi =1$ on $B\left( x_{0},R\right) $ and $%
\left\vert \nabla \psi \right\vert \leq \frac{c}{R}$. Eventually, we will
let $R\rightarrow \infty $ and then $\delta \rightarrow 0$ and $\beta
\rightarrow \infty $. The following observation will be important for this
purpose.

Using the co-area formula and that $u$ is harmonic, it was proved in Lemma
5.1 of \cite{LW3} that 
\begin{equation*}
\int_{l\left( t\right) }\left\vert \nabla u\right\vert =\int_{l\left(
s\right) }\left\vert \nabla u\right\vert <\infty ,
\end{equation*}%
for any $t,s>0$.  

By the co-area formula, it follows that for any function 
$H:\left( 0,\infty \right) \rightarrow \mathbb{R}$ we have 
\begin{equation}
\int_{L\left( a,b\right) }\left\vert \nabla u\right\vert ^{2}H\left(
u\right) =A\int_{a}^{b}H\left( s\right) ds,  \label{c}
\end{equation}%
where we have denoted with 
\begin{equation}
A=\int_{l\left( t\right) }\left\vert \nabla u\right\vert .  \label{A}
\end{equation}%
We fix $0<\delta \varepsilon <\varepsilon <T<\beta T$ and study (\ref{y3})
as $R\rightarrow \infty $. We first record two preliminary results.

\begin{lemma}
\label{0}Under the assumptions of Theorem \ref{R}, for $\chi $ given in (\ref%
{hi}), we have 
\begin{eqnarray*}
\left\vert \left( -\ln \delta \right) ^{-1}\int_{L\left( \delta \varepsilon
,\varepsilon \right) }uQ^{m}\chi ^{4m-1}-\frac{A}{4m}\left( 4m^{2}\right)
^{m-1}\right\vert &\leq &c\left( m\right) A\left( -\ln \delta \right) ^{-%
\frac{1}{2}} \\
\left\vert \left( \ln \beta \right) ^{-1}\int_{L\left( T,\beta T\right)
}uQ^{m}\chi ^{4m-1}-\frac{A}{4m}\left( 4m^{2}\right) ^{m-1}\right\vert &\leq
&c\left( m\right) A\left( \ln \beta \right) ^{-\frac{1}{2}},
\end{eqnarray*}%
for a constant $c\left( m\right) $ depending only on $m$.
\end{lemma}

\begin{proof}
By Proposition \ref{B} we have%
\begin{equation}
\left( 4m^{2}\right) ^{m-1}\int_{M}uQ\varphi ^{2m}+\mathcal{R}_{2}\leq
\int_{M}uQ^{m}\varphi ^{2m}\leq \left( 4m^{2}\right) ^{m-1}\int_{M}uQ\varphi
^{2m}+\mathcal{R}_{1}  \label{z1}
\end{equation}%
for any non-negative cut-off $\varphi $ on $M$ so that $\varphi +\left\vert
\nabla \varphi \right\vert \leq c\left( m\right) $. 

We choose $\varphi =\widetilde{\chi }^{\frac{4m-1}{2m}}\psi ^{2}$, where 
\begin{equation*}
\widetilde{\chi }\left( u\right) =\left\{ 
\begin{array}{c}
\left( -\ln \delta \right) ^{-1}\left( \ln u-\ln \left( \delta \varepsilon
\right) \right) \text{ } \\ 
\left( \ln 2\right) ^{-1}\left( \ln \left( 2\varepsilon \right) -\ln u\right)
\\ 
0%
\end{array}%
\right. 
\begin{array}{l}
\text{on }L\left( \delta \varepsilon ,\varepsilon \right) \\ 
\text{on \ }L\left( \varepsilon ,2\varepsilon \right) \\ 
\text{otherwise}%
\end{array}%
\end{equation*}
and $\psi$ is  as in (\ref{y4}).

Hence, we are applying Proposition \ref{B} to a cut-off function $\varphi $
that is supported on the end $E$, and satisfies $\varphi ^{2m}=\chi
^{4m-1}\psi ^{4m}$ on $L\left( \delta \varepsilon ,\varepsilon \right) .$ We
now want to bound $\mathcal{R}_{1}$ and $\mathcal{R}_{2}$.

We have 
\begin{eqnarray}
&&\int_{M}\left\vert \left\langle \nabla \varphi ^{2m},\nabla u\right\rangle
\right\vert Q^{m-1}  \label{w1} \\
&=&\int_{M}\left\vert \left\langle \nabla \left( \widetilde{\chi }%
^{4m-1}\psi ^{4m}\right) ,\nabla u\right\rangle \right\vert Q^{m-1}  \notag
\\
&\leq &\left( 4m-1\right) \int_{M}\left\vert \left\langle \nabla \widetilde{%
\chi },\nabla u\right\rangle \right\vert \widetilde{\chi }^{4m-2}\psi
^{4m}Q^{m-1}  \notag \\
&&+4m\int_{M}\left\vert \left\langle \nabla \psi ,\nabla u\right\rangle
\right\vert \widetilde{\chi }^{4m-1}\psi ^{4m-1}Q^{m-1}.  \notag
\end{eqnarray}%
On one hand, by (\ref{e}) we have 
\begin{eqnarray*}
\int_{M}\left\vert \left\langle \nabla \psi ,\nabla u\right\rangle
\right\vert \widetilde{\chi }^{4m-1}\psi ^{4m-1}Q^{m-1} &\leq &c\left(
m\right) \int_{\left( E\backslash B\left( x_{0},R\right) \right) \cap
L\left( \delta \varepsilon ,2\varepsilon \right) }u \\
&\leq &\frac{c\left( m\right) }{\delta \varepsilon }\int_{E\backslash
B\left( x_{0},R\right) }u^{2} \\
&\leq &\frac{C}{\delta \varepsilon }e^{-2\sqrt{\lambda _{1}\left( M\right) }%
R}.
\end{eqnarray*}%
On the other hand, by Yau's gradient estimate and (\ref{c}) we have 
\begin{eqnarray*}
\int_{M}\left\vert \left\langle \nabla \widetilde{\chi },\nabla
u\right\rangle \right\vert \widetilde{\chi }^{4m-2}\psi ^{4m}Q^{m-1} &\leq
&c\left( m\right) \int_{M}\left\vert \left\langle \nabla \widetilde{\chi }%
,\nabla u\right\rangle \right\vert \\
&=&c\left( m\right) \frac{1}{\left( -\ln \delta \right) }\int_{L\left(
\delta \varepsilon ,\varepsilon \right) }u^{-1}\left\vert \nabla
u\right\vert ^{2} \\
&&+c\left( m\right) \frac{1}{\ln 2}\int_{L\left( \varepsilon ,2\varepsilon
\right) }u^{-1}\left\vert \nabla u\right\vert ^{2} \\
&\leq &c\left( m\right) A.
\end{eqnarray*}%
Hence, by (\ref{w1}) we get as $R\rightarrow \infty ,$ 
\begin{equation}
\int_{M}\left\vert \left\langle \nabla \varphi ^{2m},\nabla u\right\rangle
\right\vert Q^{m-1}\leq c\left( m\right) A.  \label{w2}
\end{equation}%
A similar proof shows that 
\begin{equation}
\int_{M}u\left\vert \nabla \varphi \right\vert ^{2}\leq c\left( m\right) A.
\label{w3}
\end{equation}%
Following the proof of (\ref{a20'}) we obtain similarly from (\ref{a9}) for $%
k=m-2$ that%
\begin{eqnarray}
&&\int_{M}u^{-1}\left\vert u_{\alpha \bar{\beta}}\right\vert
^{2}Q^{m-2}\varphi ^{2m}+\int_{M}u^{-1}\left\vert u_{\alpha \beta
}\right\vert ^{2}Q^{m-2}\varphi ^{2m}  \label{w4} \\
&\leq &c\left( m\right) \int_{L\left( \delta \varepsilon ,2\varepsilon
\right) }u^{-1}\left\vert \nabla u\right\vert ^{2}+c\left( m\right)
\int_{M}u\left\vert \nabla \varphi \right\vert ^{2}  \notag \\
&\leq &c\left( m\right) A\left( -\ln \delta \right) .  \notag
\end{eqnarray}%
By the Cauchy-Schwarz inequality we get as $R\rightarrow \infty $ 
\begin{eqnarray*}
&&\int_{M}u^{-1}\left\vert \mathrm{Re}\left( u_{\bar{\alpha}\beta }u_{\alpha
}\left( \varphi ^{2m}\right) _{\bar{\beta}}\right) \right\vert Q^{m-2} \\
&\leq &\frac{m}{2}\int_{M}u^{-1}\left\vert u_{\bar{\alpha}\beta }\right\vert
\left\vert \nabla u\right\vert \left\vert \nabla \varphi \right\vert \varphi
^{2m-1}Q^{m-2} \\
&\leq &\frac{m}{2}\left( \int_{M}u^{-1}\left\vert u_{\alpha \bar{\beta}%
}\right\vert ^{2}Q^{m-2}\varphi ^{2m}\right) ^{\frac{1}{2}}\left(
\int_{M}u\left\vert \nabla \varphi \right\vert ^{2}Q^{m-1}\varphi
^{2m-2}\right) ^{\frac{1}{2}} \\
&\leq &c\left( m\right) \left( -\ln \delta \right) ^{\frac{1}{2}}A,
\end{eqnarray*}%
where in the last line we used (\ref{w3}) and (\ref{w4}). The estimate 
\begin{equation*}
\int_{M}u^{-1}\left\vert \mathrm{Re}\left( u_{\alpha \beta }u_{\bar{\alpha}%
}\left( \varphi ^{2m}\right) _{\bar{\beta}}\right) \right\vert Q^{m-2}\leq
c\left( m\right) \left( -\ln \delta \right) ^{\frac{1}{2}}A
\end{equation*}%
follows similarly. This proves that 
\begin{equation}
\left\vert \mathcal{R}_{i}\right\vert \leq c\left( m\right) A\left( -\ln
\delta \right) ^{\frac{1}{2}}.  \label{z2}
\end{equation}%
Making $R\rightarrow \infty $ in (\ref{z1}) and using (\ref{z2}) we get%
\begin{equation*}
\left\vert \int_{M}uQ^{m}\widetilde{\chi }^{4m-1}-\left( 4m^{2}\right)
^{m-1}\int_{M}uQ\widetilde{\chi }^{4m-1}\right\vert \leq c\left( m\right)
A\left( -\ln \delta \right) ^{\frac{1}{2}}.
\end{equation*}%
Since 
\begin{eqnarray*}
\int_{L\left( \varepsilon ,2\varepsilon \right) }uQ^{m}\widetilde{\chi }%
^{4m-1} &\leq &c\left( m\right) \int_{L\left( \varepsilon ,2\varepsilon
\right) }uQ \\
&=&c\left( m\right) A,
\end{eqnarray*}%
we conclude from above that 
\begin{equation}
\ \left\vert \int_{M}uQ^{m}\chi ^{4m-1}-\left( 4m^{2}\right)
^{m-1}\int_{M}uQ\chi ^{4m-1}\right\vert \leq c\left( m\right) A\left( -\ln
\delta \right) ^{\frac{1}{2}}.  \label{z3}
\end{equation}%
Note that by (\ref{c}) 
\begin{eqnarray}
\int_{M}uQ\chi ^{4m-1} &=&A\int_{\delta \varepsilon }^{\varepsilon }\frac{1}{%
t}\chi ^{4m-1}\left( t\right) dt \label{IntQ}\\
&=&\frac{1}{4m}\left( -\ln \delta \right) A.\notag
\end{eqnarray}
Hence, we conclude from (\ref{z3}) that 
\begin{equation*}
\left\vert \int_{L\left( \delta \varepsilon ,\varepsilon \right) }uQ^{m}\chi
^{4m-1}-\frac{A}{4m}\left( 4m^{2}\right) ^{m-1}\left( -\ln \delta \right)
\right\vert \leq c\left( m\right) A\left( -\ln \delta \right) ^{\frac{1}{2}}.
\end{equation*}%
This proves the first estimate of the lemma. The corresponding estimate on $%
L\left( T,\beta T\right) $ follows similarly, the only difference being that
we use (\ref{v}) to get 
\begin{eqnarray*}
\int_{L\left( T,\beta T\right) }\left\vert \left\langle \nabla \psi ,\nabla
u\right\rangle \right\vert &\leq &\frac{c\left( m\right) }{R}\int_{L\left(
T,\beta T\right) \cap \left( F\backslash B\left( x_{0},R\right) \right) }u \\
&\leq &c\left( m\right) \beta T\mathrm{Vol}\left( F\backslash B\left(
x_{0},R\right) \right) \\
&\leq &c\left( m\right) \beta Te^{-2\sqrt{\lambda _{1}\left( M\right) }R}.
\end{eqnarray*}%
Certainly, the right side converges to zero when $R\rightarrow \infty $ and $%
\beta ,T$ are fixed. This proves the lemma.
\end{proof}

The next result is similar to Lemma \ref{0}.

\begin{lemma}
\label{0'}Under the assumptions of Theorem \ref{R}, for $\chi $ given in (%
\ref{hi}) we have%
\begin{equation*}
\left\vert \left( -\ln \delta \right) ^{-1}\int_{L\left( \delta \varepsilon
,\varepsilon \right) }u^{-2}u_{\bar{\alpha}\beta }u_{\alpha }u_{\bar{\beta}%
}Q^{m-2}\chi ^{4m-1}-\frac{\left( m-1\right) A}{16}\left( 4m^{2}\right)
^{m-2}\right\vert \leq c\left( m\right) A\left( -\ln \delta \right) ^{-\frac{%
1}{2}}
\end{equation*}%
and 
\begin{equation*}
\left\vert \left( \ln \beta \right) ^{-1}\int_{L\left( T,\beta T\right)
}u^{-2}u_{\bar{\alpha}\beta }u_{\alpha }u_{\bar{\beta}}Q^{m-2}\chi ^{4m-1}-%
\frac{\left( m-1\right) A}{16}\left( 4m^{2}\right) ^{m-2}\right\vert \leq
c\left( m\right) A\left( \ln \beta \right) ^{-\frac{1}{2}},
\end{equation*}%
where $c\left( m\right) $ is a constant depending only on $m$.
\end{lemma}

\begin{proof}
By Proposition \ref{B} we have 
\begin{eqnarray*}
\int_{M}u^{-2}u_{\bar{\alpha}\beta }u_{\alpha }u_{\bar{\beta}}Q^{m-2}\varphi
^{2m} &\leq &\left( 4m^{2}\right) ^{m-2}\frac{m\left( m-1\right) }{4}%
\int_{M}uQ\varphi ^{2m}+\mathcal{R}_{3} \\
\int_{M}u^{-2}u_{\bar{\alpha}\beta }u_{\alpha }u_{\bar{\beta}}Q^{m-2}\varphi
^{2m} &\geq &\left( 4m^{2}\right) ^{m-2}\frac{m\left( m-1\right) }{4}%
\int_{M}uQ\varphi ^{2m}+\mathcal{R}_{4}
\end{eqnarray*}%
for any cut-off $\varphi $ on $M$. We choose the cut-off $\varphi =%
\widetilde{\chi }^{\frac{4m-1}{2m}}\psi ^{2}$ as in Lemma \ref{0}. From (\ref%
{z2}) we know that 
\begin{equation*}
\left\vert \mathcal{R}_{i}\right\vert \leq c\left( m\right) A\left( -\ln
\delta \right) ^{\frac{1}{2}}.
\end{equation*}
By (\ref{IntQ}), this proves the estimate on  $L\left( \delta \varepsilon
,\varepsilon \right) $. The proof of the estimate on $L\left( T,\beta
T\right) $ is similar.
\end{proof}

We use Lemma \ref{0} and Lemma \ref{0'} to estimate each term in (\ref{y3})
for the cut-off $\phi $ specified in (\ref{y4}).

\begin{lemma}
\label{1}Under the assumptions of Theorem \ref{R}, for $\phi $ given in (\ref%
{y4}), we have as $R\rightarrow \infty ,$ 
\begin{equation*}
\int_{M}u\left\vert \nabla \phi \right\vert ^{2}\leq c\left( m\right)
A\left( -\ln \delta \right) ^{-1}+c\left( m\right) A\left( \ln \beta \right)
^{-1},
\end{equation*}%
for a constant $c\left( m\right) $ depending only on $m$.
\end{lemma}

\begin{proof}
We have 
\begin{equation*}
\int_{M}u\left\vert \nabla \phi \right\vert ^{2}\leq 8\int_{M}u\chi ^{4}\psi
^{2}\left\vert \nabla \psi \right\vert ^{2}+8\int_{M}u\psi ^{4}\chi
^{2}\left\vert \nabla \chi \right\vert ^{2}.
\end{equation*}%
Using (\ref{c}) we have 
\begin{eqnarray}
\int_{M}u\psi ^{4}\chi ^{2}\left\vert \nabla \chi \right\vert ^{2} &\leq
&c\left( \ln \beta \right) ^{-2}\int_{L\left( T,\beta T\right)
}u^{-1}\left\vert \nabla u\right\vert ^{2}  \label{y5} \\
&&+c\left( -\ln \delta \right) ^{-2}\int_{L\left( \delta \varepsilon
,\varepsilon \right) }u^{-1}\left\vert \nabla u\right\vert ^{2}  \notag \\
&=&cA\left( -\ln \delta \right) ^{-1}+cA\left( \ln \beta \right) ^{-1}. 
\notag
\end{eqnarray}%
By (\ref{v}) and (\ref{e}) we also have that 
\begin{eqnarray}
\int_{M}u\chi ^{4}\psi ^{2}\left\vert \nabla \psi \right\vert ^{2} &\leq
&c\int_{B\left( x_{0},2R\right) \backslash B\left( x_{0},R\right) }u\chi
\label{y6} \\
&\leq &c\int_{E\backslash B\left( x_{0},R\right) }u\chi +c\int_{F\backslash
B\left( x_{0},R\right) }u\chi  \notag \\
&\leq &\frac{c}{\delta \varepsilon }\int_{E\backslash B\left( x_{0},R\right)
}u^{2}  \notag \\
&&+c\beta T\mathrm{Vol}\left( F\backslash B(x_{0},R\right) )  \notag \\
&\leq &C\left( \frac{1}{\delta \varepsilon }+\beta T\right) e^{-2\sqrt{%
\lambda _{1}\left( M\right) }R}.  \notag
\end{eqnarray}%
The result follows from (\ref{y5}) and (\ref{y6}).
\end{proof}

We continue with the following.

\begin{lemma}
\label{2}Under the assumptions of Theorem \ref{R}, for $\phi $ given by (\ref%
{y4}), we have as $R\rightarrow \infty $ 
\begin{equation*}
\left\vert \int_{M}\phi ^{2m-1}\left\langle \nabla u,\nabla \phi
\right\rangle Q^{m-1}\right\vert \leq c\left( m\right) A\left( -\ln \delta
\right) ^{-\frac{1}{2}}+c\left( m\right) A\left( \ln \beta \right) ^{-\frac{1%
}{2}},
\end{equation*}%
for a constant $c\left( m\right) $ depending only on $m$.
\end{lemma}

\begin{proof}
We have 
\begin{eqnarray}
&&\int_{M}\phi ^{2m-1}\left\langle \nabla u,\nabla \phi \right\rangle Q^{m-1}
\label{y8} \\
&=&2\int_{M}\phi ^{2m-1}\chi ^{2}\psi \left\langle \nabla u,\nabla \psi
\right\rangle Q^{m-1}  \notag \\
&&+2\int_{M}\phi ^{2m-1}\chi \psi ^{2}\left\langle \nabla u,\nabla \chi
\right\rangle Q^{m-1}.  \notag
\end{eqnarray}%
For the first term, we use Yau's gradient estimate $Q\leq c\left( m\right) $
and (\ref{y6}) to get 
\begin{eqnarray}
&&\int_{M}\phi ^{2m-1}\chi ^{2}\psi \left\vert \left\langle \nabla u,\nabla
\psi \right\rangle \right\vert Q^{m-1}  \label{y9} \\
&\leq &c\left( m\right) \int_{B\left( x_{0},2R\right) \backslash B\left(
x_{0},R\right) }u\chi  \notag \\
&\leq &C\left( \frac{1}{\delta \varepsilon }+\beta T\right) e^{-2\sqrt{%
\lambda _{1}\left( M\right) }R}.  \notag
\end{eqnarray}%
For the second term in (\ref{y8}) note that 
\begin{eqnarray*}
&&\int_{M}\phi ^{2m-1}\chi \psi ^{2}\left\langle \nabla u,\nabla \chi
\right\rangle Q^{m-1} \\
&=&\left( -\ln \delta \right) ^{-1}\int_{L\left( \delta \varepsilon
,\varepsilon \right) }uQ^{m}\chi ^{4m-1}\psi ^{4m} \\
&&-\left( \ln \beta \right) ^{-1}\int_{L\left( T,\beta T\right) }uQ^{m}\chi
^{4m-1}\psi ^{4m}.
\end{eqnarray*}%
By Lemma \ref{0} we have as $R\rightarrow \infty ,$ 
\begin{equation}
\left\vert \int_{M}\phi ^{2m-1}\chi \psi ^{2}\left\langle \nabla u,\nabla
\chi \right\rangle Q^{m-1}\right\vert \leq c\left( m\right) A\left( -\ln
\delta \right) ^{-\frac{1}{2}}+c\left( m\right) A\left( \ln \beta \right) ^{-%
\frac{1}{2}}.  \label{y10}
\end{equation}%
By (\ref{y8}), (\ref{y9}) and (\ref{y10}) we conclude as $R\rightarrow
\infty ,$ 
\begin{equation*}
\left\vert \int_{M}\phi ^{2m-1}\left\langle \nabla u,\nabla \phi
\right\rangle Q^{m-1}\right\vert \leq c\left( m\right) A\left( -\ln \delta
\right) ^{-\frac{1}{2}}+c\left( m\right) A\left( \ln \beta \right) ^{-\frac{1%
}{2}}.
\end{equation*}%
This proves the result.
\end{proof}

We now prove the following.

\begin{lemma}
\label{3}Under the assumptions of Theorem \ref{R}, for $\phi $ given in (\ref%
{y4}), we have as $R\rightarrow \infty $ 
\begin{equation*}
\left\vert \int_{M}u^{-1}\mathrm{Re}\left( u_{\bar{\alpha}\beta }u_{\alpha
}\left( \phi ^{2m}\right) _{\bar{\beta}}\right) Q^{m-2}\right\vert \leq
c\left( m\right) A\left( -\ln \delta \right) ^{-\frac{1}{2}}+c\left(
m\right) A\left( \ln \beta \right) ^{-\frac{1}{2}}
\end{equation*}%
and 
\begin{equation*}
\left\vert \int_{M}u^{-1}\mathrm{Re}\left( u_{\alpha \beta }u_{\bar{\alpha}%
}\left( \phi ^{2m}\right) _{\bar{\beta}}\right) Q^{m-2}\right\vert \leq
c\left( m\right) A\left( -\ln \delta \right) ^{-\frac{1}{2}}+c\left(
m\right) A\left( \ln \beta \right) ^{-\frac{1}{2}},
\end{equation*}%
for a constant $c\left( m\right) $ depending only on $m.$
\end{lemma}

\begin{proof}
We have 
\begin{eqnarray}
&&\int_{M}u^{-1}\mathrm{Re}\left( u_{\bar{\alpha}\beta }u_{\alpha }\phi _{%
\bar{\beta}}\right) \phi ^{2m-1}Q^{m-2}  \label{y11} \\
&=&2\int_{M}u^{-1}\left( u_{\bar{\alpha}\beta }u_{\alpha }\chi _{\bar{\beta}%
}\right) \chi \psi ^{2}\phi ^{2m-1}Q^{m-2}  \notag \\
&&+2\int_{M}u^{-1}\left( u_{\bar{\alpha}\beta }u_{\alpha }\psi _{\bar{\beta}%
}\right) \chi ^{2}\psi \phi ^{2m-1}Q^{m-2}.  \notag
\end{eqnarray}%
As in the proof of Lemma \ref{0}, we use (\ref{a20'}), (\ref{e}) and (\ref{v}%
) to estimate%
\begin{eqnarray}
&&\left\vert \int_{M}u^{-1}\left( u_{\bar{\alpha}\beta }u_{\alpha }\psi _{%
\bar{\beta}}\right) \chi ^{2}\psi \phi ^{2m-1}Q^{m-2}\right\vert  \label{y12}
\\
&\leq &\frac{1}{4}\int_{M}u^{-1}\left\vert u_{\bar{\alpha}\beta }\right\vert
\left\vert \nabla u\right\vert \left\vert \nabla \psi \right\vert \chi
^{2}\psi \phi ^{2m-1}Q^{m-2}  \notag \\
&\leq &\frac{1}{4}\left( \int_{M}u^{-1}\left\vert u_{\bar{\alpha}\beta
}\right\vert ^{2}Q^{m-2}\phi ^{2m}\right) ^{\frac{1}{2}}\left(
\int_{M}uQ^{m-1}\left\vert \nabla \psi \right\vert ^{2}\phi ^{2m-2}\chi
^{2}\right) ^{\frac{1}{2}}  \notag \\
&\leq &c\left( \frac{1}{\delta \varepsilon }+\beta T\right) e^{-\sqrt{%
\lambda _{1}\left( M\right) }R}.  \notag
\end{eqnarray}%
Furthermore, we have 
\begin{eqnarray*}
&&\int_{M}u^{-1}\left( u_{\bar{\alpha}\beta }u_{\alpha }\chi _{\bar{\beta}%
}\right) \chi \psi ^{2}\phi ^{2m-1}Q^{m-2} \\
&=&\left( -\ln \delta \right) ^{-1}\int_{L\left( \delta \varepsilon
,\varepsilon \right) }u^{-2}u_{\bar{\alpha}\beta }u_{\alpha }u_{\bar{\beta}%
}Q^{m-2}\chi ^{4m-1}\psi ^{4m} \\
&&-\left( \ln \beta \right) ^{-1}\int_{L\left( T,\beta T\right) }u^{-2}u_{%
\bar{\alpha}\beta }u_{\alpha }u_{\bar{\beta}}Q^{m-2}\chi ^{4m-1}\psi ^{4m}.
\end{eqnarray*}%
Using Lemma \ref{0'} it follows that 
\begin{eqnarray}
&&\left\vert \int_{M}u^{-1}\left( u_{\bar{\alpha}\beta }u_{\alpha }\chi _{%
\bar{\beta}}\right) Q^{m-2}\chi ^{4m-1}\psi ^{4m}\right\vert  \label{y13} \\
&\leq &c\left( m\right) A\left( -\ln \delta \right) ^{-\frac{1}{2}}+c\left(
m\right) A\left( \ln \beta \right) ^{-\frac{1}{2}}.  \notag
\end{eqnarray}%
Plugging (\ref{y13}) and (\ref{y12}) into (\ref{y11}) implies as $%
R\rightarrow \infty $%
\begin{equation*}
\left\vert \int_{M}u^{-1}\mathrm{Re}\left( u_{\bar{\alpha}\beta }u_{\alpha
}\phi _{\bar{\beta}}\right) \phi ^{2m-1}Q^{m-2}\right\vert \leq c\left(
m\right) A\left( -\ln \delta \right) ^{-\frac{1}{2}}+c\left( m\right)
A\left( \ln \beta \right) ^{-\frac{1}{2}}.
\end{equation*}%
This proves the first estimate of the lemma. For the other estimate, we
proceed similarly and get 
\begin{eqnarray*}
&&\left\vert \int_{M}u^{-1}\mathrm{Re}\left( u_{\alpha \beta }u_{\bar{\alpha}%
}\phi _{\bar{\beta}}\right) Q^{m-2}\phi ^{2m-1}\right\vert \\
&\leq &2\left\vert \int_{M}u^{-1}\mathrm{Re}\left( u_{\alpha \beta }u_{\bar{%
\alpha}}\chi _{\bar{\beta}}\right) \chi \psi ^{2}\phi
^{2m-1}Q^{m-2}\right\vert \\
&&+2\left\vert \int_{M}u^{-1}\mathrm{Re}\left( u_{\alpha \beta }u_{\bar{%
\alpha}}\psi _{\bar{\beta}}\right) \chi ^{2}\psi \phi
^{2m-1}Q^{m-2}\right\vert \\
&\leq &2\left\vert \int_{M}u^{-1}\mathrm{Re}\left( u_{\alpha \beta }u_{\bar{%
\alpha}}\chi _{\bar{\beta}}\right) \chi \psi ^{2}\phi
^{2m-1}Q^{m-2}\right\vert \\
&&+C\left( \frac{1}{\delta \varepsilon }+\beta T\right) e^{-\sqrt{\lambda
_{1}\left( M\right) }R},
\end{eqnarray*}%
where the last line follows similarly to (\ref{y12}). Furthermore, we have 
\begin{eqnarray*}
&&\int_{M}u^{-1}\mathrm{Re}\left( u_{\alpha \beta }u_{\bar{\alpha}}\chi _{%
\bar{\beta}}\right) \chi \psi ^{2}\phi ^{2m-1}Q^{m-2} \\
&=&\frac{1}{\left( -\ln \delta \right) }\int_{L\left( \delta \varepsilon
,\varepsilon \right) }u^{-2}\mathrm{Re}\left( u_{\alpha \beta }u_{\bar{\alpha%
}}u_{\bar{\beta}}\right) \chi ^{4m-1}\psi ^{4m}Q^{m-2} \\
&&-\frac{1}{\ln \beta }\int_{L\left( T,\beta T\right) }u^{-2}\mathrm{Re}%
\left( u_{\alpha \beta }u_{\bar{\alpha}}u_{\bar{\beta}}\right) \chi
^{4m-1}\psi ^{4m}Q^{m-2}.
\end{eqnarray*}%
Using a cut-off function $\varphi $ as in Lemma \ref{0}, and by (\ref{m5}),
Lemma \ref{0} and Lemma \ref{0'} it follows that as $R\rightarrow \infty ,$%
\begin{equation*}
\left\vert \int_{M}u^{-1}\mathrm{Re}\left( u_{\alpha \beta }u_{\bar{\alpha}%
}\chi _{\bar{\beta}}\right) \chi \psi ^{2}\phi ^{2m-1}Q^{m-2}\right\vert
\leq c\left( m\right) A\left( -\ln \delta \right) ^{-\frac{1}{2}}+c\left(
m\right) A\left( \ln \beta \right) ^{-\frac{1}{2}}.
\end{equation*}%
We therefore obtain from above that as $R\rightarrow \infty ,$ 
\begin{equation*}
\left\vert \int_{M}u^{-1}\mathrm{Re}\left( u_{\alpha \beta }u_{\bar{\alpha}%
}\left( \phi ^{2m}\right) _{\bar{\beta}}\right) Q^{m-2}\right\vert \leq
c\left( m\right) A\left( -\ln \delta \right) ^{-\frac{1}{2}}+c\left(
m\right) A\left( \ln \beta \right) ^{-\frac{1}{2}}.
\end{equation*}%
This proves the lemma.
\end{proof}

We now finish the proof of Theorem \ref{R}. By (\ref{y3}) and Lemmas \ref{1}%
, \ref{2} and \ref{3}, we get as $R\rightarrow \infty $ 
\begin{equation*}
\left\vert \int_{M}uQ^{m}\phi ^{2m}-\left( 4m^{2}\right)
^{m-1}\int_{M}uQ\phi ^{2m}\right\vert \leq c\left( m\right) A\left( -\ln
\delta \right) ^{-\frac{1}{2}}+c\left( m\right) A\left( \ln \beta \right) ^{-%
\frac{1}{2}},
\end{equation*}%
where $A=\int_{l\left( t\right) }\left\vert \nabla u\right\vert <\infty $.
Making $\delta \rightarrow 0$ and $\beta \rightarrow \infty $ implies that
all inequalities used in proving (\ref{I}) in Proposition \ref{B} must turn
into equalities. Now (\ref{m10}) implies that 
\begin{equation}
\left\vert \nabla \ln u\right\vert =2m\text{,}  \label{k1}
\end{equation}%
and by (\ref{a13}) we have that 
\begin{equation}
u_{\bar{\alpha}\beta }-u^{-1}u_{\bar{\alpha}}u_{\beta }+mg_{\bar{\alpha}%
\beta }u=0.  \label{k2}
\end{equation}%
Moreover, equality in (\ref{m6}) yields 
\begin{eqnarray}
\left\vert u_{\alpha \beta }u_{\bar{\alpha}}u_{\bar{\beta}}\right\vert &=&%
\frac{1}{4}\left\vert u_{\alpha \beta }\right\vert \left\vert \nabla
u\right\vert ^{2}  \label{k3} \\
\left\vert u_{\alpha \beta }\right\vert &=&m\left( m+1\right) u.  \notag
\end{eqnarray}

Note that (\ref{k1}), (\ref{k2}) and (\ref{k3}) are the same as the
identities (12) in the proof of Theorem 4 of \cite{M}. They imply the
splitting as claimed in Theorem \ref{R}. This proves the result.

\end{document}